\documentclass[11 pt]{amsart}
\usepackage[all]{xy}
\usepackage{amssymb,amsmath,color}
\usepackage{amsfonts}
\usepackage{amssymb}
\usepackage{amscd}
\usepackage{verbatim}
\usepackage{latexsym}
\usepackage{tikz}
\usetikzlibrary{matrix,arrows}
\usepackage{amsmath,color}
\usepackage{hyperref}

\newtheorem{theorem}{Theorem}

\newtheorem{question}{Question}

\newtheorem{lemma}[theorem]{Lemma}

\newtheorem{cor}[theorem]{Corollary}

\newtheorem{Def}[theorem]{Definition}
\newtheorem{prop}[theorem]{Proposition}
\newtheorem{fact}[theorem]{Fact}

\theoremstyle{remark}

\numberwithin{equation}{section}

\def\ra{\rightarrow}
\def\rhob{ {\bar {\rho} } }
\def\rhobar{ {\bar {\rho} } }

\def\FF{{\mathbb F}}

\def\QQ{{\mathbb Q}}
\def\Q{{\mathbb Q}}
\def\ZZ{{\mathbb Z}}
\def\Z{{\mathbb Z}}

\def\TT{{\mathbb T}}
\def\T{{\mathbb T}}

\def\LP{ \left(\begin{array}}
\def\RP{ \end{array}\right)}
\def\LB{ \left[\begin{array}}
\def\RB{ \end{array}\right]}

\def\QQ{ {\mathbb Q}}

\def\eps{\epsilon}
\def\rhob{ {\bar {\rho} } }
\def\FF{ {\mathbb F} }

\def\wfq{{  W(\mathbb F_{p^f}) }}
\def\ad{Ad^0\bar{\rho}}

\def\mco{\mathcal O}
\def\Oc{\mathcal O}
\def\mcl{\mathcal L}
\def\mclt{ \tilde{\mathcal L} }

\def\epsb{ \bar{\epsilon}}

\def\gal{G_{\mathbb Q}}

\DeclareFontFamily{U}{wncy}{}
\DeclareFontShape{U}{wncy}{m}{n}{<->wncyr10}{}
\DeclareSymbolFont{mcy}{U}{wncy}{m}{n}
\DeclareMathSymbol{\Sh}{\mathord}{mcy}{"58}

\catcode`,\active

\catcode`\,12

\title{Lifting  torsion Galois  representations}
\author{Chandrashekhar Khare and Ravi Ramakrishna}
\address{\newline
Department of Mathematics, UCLA, Los Angeles, CA 90095-1555, USA \newline
Department of Mathematics, Cornell University, Ithaca, NY 14853-4210, USA}
\email{shekhar@math.ucla.edu, ravi@math.cornell.edu}
\thanks{The 
authors would like to thank the Tata Institute of Fundamental Research for its hospitality
while this paper was started. 
The first author was supported by NSF grant DMS-1161671 and by a Humboldt Research Award.
The second author also thanks the African Institute for Mathematical Sciences
in Muizenberg, South Africa for its hospitality.}

\begin{document}
\maketitle
\begin{abstract} Let $p$ be an odd prime and
$\mco/\ZZ_p$ be of degree $d=ef$  with uniformiser $\pi$ and residue field
$\FF_{p^f}$. Let
$\rho_n:\gal \to GL_2\left(\mco/(\pi^n)\right)$ 
have modular mod $\pi$ reduction $\rhob$, be ordinary at $p$,  and satisfy some mild technical conditions. We show that $\rho_n$ can be lifted to a characteristic 0 geometric representation which arises from a newform. Earlier methods could not handle the case  of $n>1$ when $e>1$.

 We  also show  that a prescribed totally ramified complete discrete valuation ring
$\mco$ is the weight $2$ deformation ring for $\rhob$ (or the semistable or flat quotient of this ring in a few cases)
for a suitable choice of auxiliary level.  
This implies
the field of Fourier coefficients  of newforms of square free level that give rise to $\rhobar$ can
have arbitrarily large ramification index  at $p$.

In  an appendix,  we use  some of the work in the paper to prove  by  $p$-adic approximation,  the  modularity  of  many ordinary,  geometric,   finitely ramified representations $\rho:G_\Q \rightarrow GL_2(\Oc)$ assuming the reduction is modular and irreducible. The previous methods  along these lines  in \cite{K} applied only for $\mco/\ZZ_p$ unramified.  In the appendix we are recovering  known  results of Wiles and Taylor-Wiles  by a different method.

\end{abstract}

\section{Introduction}

Let $\rhobar: G_\Q \rightarrow GL_2(k)$ be a continuous, odd, irreducible representation of the absolute Galois group $G_\Q$ of $\Q$  with $k$ a finite field of characteristic $p$. We call such a representation of $S$-type. 

We assume $p$ is odd.
The second  author  introduced in \cite{R2} a method of  killing dual Selmer groups 
to show that there were geometric lifts (i.e. ramified at only finitely many primes, and deRham at $p$) of  (most) $\rhobar$  of $S$-type, with ramification allowed at a set of  auxiliary primes $Q$. 
(Also see \cite{T} where the language of dual Selmer groups was introduced.)
 The goal of that work was to show that $\rhobar$ did have in the first place  a characteristic zero lift without assuming that $\rhobar$ was modular.   
Let  $W(k)$ is the ring of Witt vectors. 
Then the same method applied to lifting representations
 $\rho_n:G_\Q \rightarrow GL_2(W(k)/p^n))$, with reduction $\rhobar$ of $S$-type, to geometric characteristic 0 representations if $p>3$ provided $\rho_n$ is balanced in a sense made precise  below. 
 
  In this paper we address the question of producing geometric lifts of  representations $\rho_n:G_\Q \rightarrow GL_2(\Oc/(\pi^n))$ with reduction $\rhobar$ where $\Oc$ is the ring of integers of a ramified extension of $\Q_p$. The method of \cite{R2} does not work 
directly as  when  $\Oc/\ZZ_p$ ramified, the map
$GL_2(\Oc/(\pi^2)) \ra GL_2(\Oc/(\pi))$ is split. We 
do this under the conditions that $\rho_n$ is balanced, ordinary and has full image (cf. \S\ref{one}).

There are two steps. In Theorem \ref{easy} we firstly show that $\rho_n$ lifts to a characteristic 0, ordinary representation which is finitely ramified. Using this, in Theorem \ref{hard}, for which we only consider the case when $\mco$ is totally ramified,  we show in fact that we can also get lifts with the further property of being geometric at $p$.

For this purpose, we first  study  the ordinary (no weight is fixed) deformation  rings $R^{ord, Q-new}$,  and  construct  sets of primes $Q$ such that $R^{ord, Q-new}$ has $1$-dimensional tangent space and thus is a quotient of $\wfq[[U]]$.  Further $\rho_n$ arises as a specialization of the universal representation 
associated to  $R^{ord, Q-new}$. As  a certain dual Selmer group vanishes, we in fact deduce that  $R^{ord, Q-new} \simeq \wfq[[U]]$. This already ensures that $\rho_n$ lifts to an ordinary characteristic 0 representation which is ramified at finitely many primes (cf. Theorem \ref{easy}).  We also deduce that $R^{ord, Q-new}$ is isomorphic to  the related  Hida Hecke algebra $\T^{ord, Q-new}$, as the latter is  known to be non-zero and flat over $\Z_p[[T]]$ by level  raising results of \cite{DT} and Hida's 
results on the structure of  ordinary Hecke algebras respectively.

In Theorem \ref{hard} we further show that   by a more careful choice of $Q$, 
the weight $2$  quotient of  $R^{ord, Q-new}$  (or more precisely  the semistable or flat quotient of this weight 2 ring  in a few cases) can 
be  controlled to be  $\Oc$ and  $\rho_n$ arises from this weight 2 quotient.   
The proof of Theorem \ref{hard}  is   more involved than the proof of Theorem \ref{easy}. Essential use is made of the techniques of \cite{KLR}.

In \cite{R2} no assumption of modularity on $\rhobar$ was made. In proving Theorem~\ref{hard}, we only use modularity of $\rhobar$ in one place 
(cf. Lemma  \ref{disting}).  One of the interests of this theorem is that we manage to lift  the representation $\rho_n$ to a  characteristic 0 geometric lift (of weight 2), in spite of it being impossible to  kill the   minimal weight $2$ dual Selmer (and Selmer)  group of $\rhobar$  using primes which are nice for  
$\rho_n$ (not just $\rhob$!) when $\Oc/\ZZ_p$ is  ramified.   

This brings to mind a   folklore question:

\begin{question} Let $K$ be an imaginary quadratic field and let $\rhobar:G_K \rightarrow GL_2(k)$
be irreducible.
Do there always exist  characteristic $0$ geometric lifts of $\rhob$? 
\end{question} 

 Indeed its hard to guess an answer  either way with any confidence!  Here too a principal obstacle to lifting  is that a certain  dual Selmer group  cannot be killed. Presently we do not see  if our methods will help in answering the question affirmatively.

In the appendix, we  use the isomorphisms $ \wfq[[U]] \simeq R^{ord, Q-new} \simeq   \T^{ord, Q-new}$  above to extend the  approach of \cite{K} to modularity of geometric lifts $\rho$ of $\rhobar$ by $p$-adic approximation, without imposing the condition that $\rho$ is defined over the Witt vectors. This condition was essential in \cite{K} again because $GL_2(\Oc/(\pi^2)) \ra GL_2(\Oc/(\pi))$ is split if $\Oc/\Z_p$ is ramified. For  a similar reason it was assumed in \cite{K} that $p>3$ (as the homomorphism $GL_2(\Z/9) \rightarrow GL_2(\Z/3)$ is split),  an assumption that we remove in the appendix.  The appendix uses  only Theorem \ref{easy} of this paper (and not the more involved  Theorem \ref{hard}), together with level lowering arguments. We also note that Jack Thorne in \cite{Thorne} has used the strategy  of \cite{K} to prove modularity lifting theorems in new cases.

In the paper we have not striven for generality, and  have opted to  present the new ideas of this paper in their simplest setting.

{\sf Acknowledgements:}  Much of \S $2$, \S $3$ and Theorem~\ref{lundell} of this paper is unpublished work of Benjamin Lundell in his Cornell Ph.D. thesis, \cite{L}. The results presented here are slightly more general. The  authors thanks G.~B\"ockle,  B. Lundell and  M. Stillman for helpful conversations.


\section{The setup}\label{one}

Let $p \geq 3$ be a prime and
$\mco$ be the ring of integers of a finite   extension of $\QQ_p$ with uniformiser $\pi$
and residue field $k=\FF_{p^f}$, a finite extension of $\FF_p$.
Suppose $[\mco:\ZZ_p]=d=ef$ with $e,f$ the ramification and inertial degrees. Let $\eps$ be the $p$-adic cyclotomic character and let $\epsb$ be its mod $p$ reduction.

We consider  representations $\rhobar:G_\Q \rightarrow GL_2(k) $ of $S$-type with $k$ of odd characteristic and lifts $\rho_n: G_\Q \rightarrow GL_2(\Oc/(\pi^n))$ and $\rho:G_\Q \rightarrow GL_2(\Oc)$ of $\rhobar$. We assume that these representations are
odd, ordinary at $p$, have full image (i.e. the image contains $SL_2(k)$, $SL_2(\Oc/(\pi^n))$ and $SL_2(\Oc)$ respectively) and have determinant $\eps$, that $\rhob$ is modular and that all
these representations are of weight 2, by which we mean that their restriction to an inertia group $I_p$ at $p$ is conjugate to $\left(\begin{array}{cc}\eps & *\\ 0 & 1\end{array}\right)$, with $\eps$ the cyclotomic character. \footnote{This assumption is mainly for convenience and we could have gotten by assuming that the representations  restricted to $I_p$ are of the form  $\left(\begin{array}{cc}\eps^{k-1} & *\\ 0 & 1\end{array}\right)$ with $k \geq 2$ and $\rhobar$ is distinguished at $p$, with the further conditions on $\rho_n$ below when $k=2$  and $\rhobar$ is split at $p$. The assumption on the determinant is also of a similar nature.}

At $p$ we require  that 
\begin{itemize} 
\item If $\rhob\mid_{G_p} =\left(\begin{array}{cc}\epsb & *\\ 0 & 1\end{array}\right)$ is 
flat then $\rho_n \mid_{G_p}$ is either  flat or semistable, that is if $\rho_n \mid_{G_p}$ is not flat,
$\rho_n\mid_{G_p} =\left(\begin{array}{cc}\eps & *\\ 0 & 1\end{array}\right)$ and the $*$ arises
from taking a $p$-power root of a nonunit of $\ZZ_p$.
\item If $\rhob\mid_{G_p} =\left(\begin{array}{cc}\epsb & 0\\ 0 & 1\end{array}\right)$
then  $\rho_n\mid_{G_p}$ is   flat.
\end{itemize}
The first condition on $\rho_n \mid_{G_p}$ is simply requiring that it
lift to characteristic $0$ exist and is therefore necessary .
 The second condition is more restrictive.  
We do not know how to deal with the second case above if  $\rho_n \mid_{G_p}$ is not flat.

Let $S_0$ and $S$  be the sets
of ramified primes of $\rhob$ and $\rho_n$ respectively (these include $p$). 

We assume $\rhob$ is ordinary and modular of weight $2$  (hence odd)  and trivial nebentype
so $S\supset S_0 \supset \{p,\infty\}$.
We also assume that $\rho_n$ 
is {\bf balanced} at weight $2$. 
To understand this last 
condition, which involves properties of $\rho_n$ at $v\neq p$,  we recall some basics on Selmer and dual Selmer groups.

 For each $v\in S_0$ a smooth
quotient of the versal weight $2$ deformation ring of $\rhob \mid_{G_v}$ has been defined 
on pages $120$-$124$ of \cite{R2} and in \cite{T}. 
The  points of (the $Spec$ of) this smooth quotient are our allowable deformations and are denoted ${\mathcal C}_v$.
Corresponding to the tangent space of  this smooth quotient is a subspace $\mcl_v \subset H^1(G_v,\ad)$. 
Fact~\ref{h0ad0} follows from the discussion in \cite{R2} referred to above.

\begin{fact} \label{h0ad0} For all $v \in S_0$  
there exist ${\mathcal C}_v$ and $\mcl_v$ as above satisfying
$\dim \mcl_v =\dim H^0(G_v,\ad)+\delta_{vp}$, where $\delta_{vp}=0$ or $1$ depending on $v \neq p$ and $v=p$.
\end{fact}

Let $M$ be an $\FF_{p^f}[\gal]$-module with ${\mathbb G}_m$-dual $M^*$ and let $R$ be the union of the places whose inertial
action on $M$ is nontrivial and $\{p,\infty\}$.
Let ${\mathcal M}_v \subset H^1(G_v,M)$ with annihilator
${\mathcal M}^{\perp}_v \subset H^1(G_v,M^*)$ under the perfect local pairing 
$$H^1(G_v,M) \times H^1(G_v,M^*) \to H^2(G_v,\FF_{p^f}(1)) \simeq \FF_{p^f}.$$
Set the Selmer group for the subspaces ${\mathcal M}_v \subset H^1(G_v,M)$ to be
$$H^1_{\mathcal M}(G_R,M):=\mbox{Ker}\left( H^1(G_R,M)  \to \oplus_{v \in R} \frac{H^1(G_v,M)}{{\mathcal M}_v}\right)$$
and the dual Selmer group
$$H^1_{{\mathcal M}^{\perp}}(G_R,M^*):=\mbox{Ker}\left( H^1(G_R,M^*)  \to \oplus_{v \in R} 
\frac{H^1(G_v,M^*)}{{\mathcal M}^{\perp}_v}\right)$$
 
Recall Proposition $1.6$ of \cite{W}:
\begin{prop}\label{wiles}
 $$\dim H^1_{\mathcal M}(G_R,M) -\dim H^1_{{\mathcal M}^{\perp}}(G_R,M^*)$$
$$=\dim H^0(\gal,M) -\dim H^0(\gal,M^*) + \sum_{v\in R}  
\left(\dim {\mathcal M}_v - \dim H^0(G_v,M)\right).$$
\end{prop}

Fact~\ref{h0ad0}, Proposition~\ref{wiles} and the fact that
$\rhob$ is odd with full image  together imply the ordinary weight $2$
Selmer group 
$H^1_{\mathcal L}(G_{S_0},\ad)$
and its dual Selmer group
$H^1_{{\mathcal L}^{\perp}}(G_{S_0},\ad)$
have the same
dimension. We need the following 
\vskip1em

\noindent
{\bf Balancedness Assumption.} Let $v \in S\backslash S_0$.
We assume a smooth quotient 
of the versal deformation ring for $\rhob\mid_{G_v}$ exists with points ${\mathcal C}_v$ and
induced subspace
$\mcl_v \subset H^1(G_v,\ad)$ such that $\dim H^0(G_v,\ad) = \dim \mcl_v$.
\vskip1em

The local conditions $\mcl_v$ for $v\in S_0$  are now standard so we do not recall
their definitions. See \cite{R2} and \cite{T}. We will be doing ordinary at $p$ deformation theory with 
arbitrary weights so we will also work with $Ad\rhob$, the full adjoint, as well.
We define local conditions
 ${\tilde {\mathcal L}}_v $ for the full adjoint.
\begin{Def} \label{tildes}
1) For $v\neq p$ set
$${\tilde {\mathcal L}}_v := {\mathcal L}_v \oplus H^1_{nr}(G_v,\FF_{p^f}) 
\subset H^1(G_v,\ad) \oplus H^1(G_v,\FF_{p^f})=H^1(G_v,Ad\rhob),$$
that is the direct sum of ${\mathcal L}_v$ and 
and the $\FF_{p^f}$-valued unramified twists in the dual numbers. 
Set $\tilde{\mathcal C}_v$ to be all unramified twists of the points ${\mathcal C}_v$.
\newline\noindent
2) For $v=p$,
set $W = \left(\begin{array}{cc} *& *\\0 & 0\end{array}\right)$ and
${\tilde {\mathcal L}}_p = \mbox{Ker} \left(H^1(G_p,Ad\rhob) \to H^1(I_p,Ad\rhob/W)\right)$. Then
$\mclt_p \supset \mcl_p$ and
$\mclt_p \supset H^1_{nr}(G_p,\FF_{p^f})$, the unramified twists, though it need not be the direct
sum of these subspaces.
Set $\tilde{\mathcal C}_p$ to be the ordinary deformations of $\rhob$ of  any
weight.
\end{Def}

\section{Local deformation rings}

\subsection{Ordinary deformation rings at $p$}

We need 
Lemmas~\ref{units}  and~\ref{fontaine} for $5$) of Proposition~\ref{vequalsp}.

\begin{lemma} \label{units} 
Let $L$ be the composite of $\QQ_p(\mu_{p^{\infty}})$ and the $\ZZ_p$-unramified extension
of $\QQ_p$.
Let $\psi:G_p \to \ZZ^*_p$ be a character.
Then there exists a nonsplit representation
$\rho_{\psi} :G_p \to GL_2(\ZZ_p)$ with $\rho_{\psi}=
\left(\begin{array}{cc} \psi & *\\ 0 & 1\end{array}\right)$ where $* \equiv 0$ mod $p$ but
$* \not \equiv 0$ mod $p^2$.
Furthermore, after base change to $L$, the $*$ arises, via Kummer theory, by taking $p$-power
roots of units of elements of $L$.
\end{lemma}
\begin{proof} That $\rho_{\psi}$ exists follows from the well-known fact that
$$\dim H^1(G_p, \QQ_p(\psi)) = \left\{ \begin{array}{cc} 2 & \psi = \eps \\ 1 & \psi \neq \eps
\end{array}\right..$$
Simply take $f \in H^1(G_p,\QQ_p(\psi))$,  consider
$\rho_f:G_p \to GL_2(\QQ_p)$ given by 
$\rho(\tau) =\left(\begin{array}{cc} \psi(\tau) & f(\tau) \\0 & 1\end{array}\right)$ and
conjugate by an appropriate power of 
$\left(\begin{array}{cc} p&0\\0&1\end{array}\right)$ to get the desired integral representation.
Rather than deal with $\rho_{\psi}$, we deal with its mod $p^m$ reduction, $\rho_{\psi,m}$.
 Let $L_m$ be the composite of $\QQ_p(\mu_{p^{m}})$ and the $\ZZ/p^{m-1}$-unramified extension
of $\QQ_p$.
Note 
the $*$ in $\rho_{\psi,m}$ gives rise to a cyclic  extension of order $p^{m-1}$, {\it not} order
$p^m$.

We treat the $\psi=\eps$ case separately. Here we can just take $\rho_{\eps,m}$ to arise
from the splitting field of $x^{p^m}-(1+p)^{p^{m-1}}$.

For $\psi \neq \eps$, let $d$ be the unique  integer satisfying $\psi \equiv \eps$ mod $p^d$
and $\psi \not \equiv \eps$ mod $p^{d+1}$.
Let $D_m$ be the maximal abelian
extension of $L_m$ whose Galois group is killed by   $p^{m-1}$. Let $K_m$ be the composite
of $L_m$ and the field fixed by Kernel$(\rho_{\psi,m})$.
\vskip1em
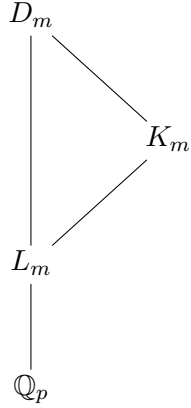
\noindent
\hspace{5cm}
\begin{tikzpicture}[description/.style={fill=white,inner sep=2pt}]
\matrix (m) [matrix of math nodes, row sep=3em,
column sep=2.5em, text height=1.2ex, text depth=0.25ex]
{ & D_m & \\
&  & K_m\\
&  L_m & \\ 
& \QQ_p & \\};
\path (m-4-2) edge  (m-3-2);
\path (m-3-2) edge  (m-2-3);
\path(m-2-3) edge   (m-1-2);
\path(m-3-2) edge   (m-1-2);
\end{tikzpicture}
\newline
\noindent
 The Kummer pairing
$Gal(D_m/L_m) \times L_m^*/{L_m^*}^{p^{m-1}} \to \mu_{p^{m-1}}$ is perfect and $Gal(L_m/\QQ_p)$-equivariant
so $Gal(D_m/L_m)$ is isomorphic to the ${\mathbb G}_m$-dual of $L_m^*/{L_m^*}^{p^{m-1}}$
as a $Gal(L_m/\QQ_p)$-module.
As $K_m/L_m$ is a cyclic  extension of order $p^{m-1}$ and $K_m/\QQ_p$ is Galois
with
 $Gal(L_m/\QQ_p)$ acting on $Gal(K_m/L_m)$ by the character $\psi$, we see
$K_m$ arises over $L_m$ by adding $p^{m-1}$st roots of an element $\alpha \in 
L_m^*/(L_m^*)^{p^{m-1}}$
which, by the above ${\mathbb G}_m$-duality, 
generates a $\eps/\psi$-eigenspace in this group. So we need to prove such an eigenspace exists
in the {\it unit} part of $L_m^*/(L_m^*)^{p^{m-1}}$.

Recall
$L_m^* \simeq \left<\pi_m\right> \times  U_{L_m}$
where $\pi_m$ is a uniformiser of $L_m$ and
 $U_{L_m}$ is the group of units.
Write $\alpha =\pi_m^{p^ka}u$ where $p\nmid a$ and $u \in U_{L_m}$
so $K_m=L_m(\alpha^{1/p^{m-1}})$.
Let $\sigma \in Gal(L_m/\QQ_p)$ and set $\sigma(\pi_m)=\pi_mw_{\sigma}$  
and $\sigma(u)=u_{\sigma}$  where 
$w_{\sigma}, u_{\sigma} \in U_{L_m}$.
We have 
$$\sigma(\alpha) =\sigma(\pi_m^{p^ka}u) = \sigma(\pi_m)^{p^ka}\sigma(u)=
\pi_m^{p^ka}w^{p^ka}_{\sigma}u_{\sigma}.$$
But we also have
$$\displaystyle\sigma(\alpha) \equiv (\alpha)^{\frac{\eps}{\psi}(\sigma)} \equiv
(\pi^{p^ka}_mu)^{\frac{\eps}{\psi}(\sigma)} \mbox{ mod } (L^*_m)^{p^{m-1}}$$
so we get 
$$(\pi^{p^ka}_m)^{\frac{\eps}{\psi}(\sigma)-1} \equiv \mbox{a unit mod } (L^*_m)^{p^{m-1}}.$$
This can only happen if the left side is trivial, that is the exponent of $\pi_m$ is a multiple
of $p^{m-1}$. Thus
$$p^k \left(\frac{\eps}{\psi}(\sigma)-1\right) \equiv 0 \mbox{ mod } p^{m-1}.$$
Since $\sigma$ is arbitrary,  the definition of $d$ implies $k+d \geq m-1$ so $k \geq m-d-1$.
Thus when we take the $p^{m-d-1}$st root of $\alpha$ and adjoin this to $L_m$, we are taking
the root of a unit. So $\rho_{\psi,m-1-d}$ arises as desired, by taking the $p$-power root of a unit. 
Now simply let $m\to \infty$.
\end{proof}
It is possible to build mod $p^m$ representations that are extensions of $1$ by $\psi$ that
do  arise by taking $p^{m-1}$st roots of nonunits of $L_m$. The proof of Lemma~\ref{units}
shows such extensions, however, do {\it not}
lift to characteristic zero when $\psi \neq \eps$.

\begin{lemma}\label{fontaine}Let the hypotheses be as in Lemma~\ref{units} and
let $\psi=\eps$. Then Kernel$(\rho_{\eps,m})$ fixes the splitting field of $x^{p^{m}} -a$
for some $a \in \ZZ_p$. This Galois representation
corresponds to finite flat group scheme over $\ZZ_p$ if and only if $a \in \ZZ^*_p$.
\end{lemma}
\begin{proof} We only sketch the proof.

It is (again) well-known that $H^1(G_p,\ZZ_p/(p^m) (\eps) ) \simeq (\ZZ_p/(p^m))^2$. The representations
associated to the splitting fields of 
$x^{p^{m}}-p$ and $x^{p^m}-(1+p)$ correspond to cohomology classes that form a basis for this module.

Since $* \equiv 0$ mod $p$ and $* \not \equiv 0$ mod $p^2$ we have that
$a$ is a $p$th power in $\ZZ_p$ but not a $p^2$th power.

If $a$ is a unit it is clear that the $G_p$-module corresponding to $\rho_{\eps,m}$ comes
from a finite flat group scheme over $\ZZ_p$.

Using Fontaine-Lafaille theory, \cite{FL}, one can count how many extensions of $\ZZ/(p^m) (\eps)$
by $\ZZ/(p^m)$ there are in the category of finite flat group schemes over $\ZZ_p$
up to isomorphism. One finds $m-1$ of them where the $*$ is as above.
These correspond to $a=(1+p)^p, (1+p)^{p^2}, \dots ,(1+p)^{p^{m-1}}$, all of which are units.
\end{proof}

\begin{prop} \label{vequalsp} 
Let
 $\bar \eta:G_p \to \FF^*_q$ be a nontrivial unramified character. 
Set $h^0 =\dim H^0(G_p,Ad\rhob)$.
Up to twist we have the the following
possibilities for $\rhob \mid_{G_p}$ and its {\bf local} deformation ring: 
\newline\noindent
1) $\rhob \mid_{G_p}=\left(\begin{array}{cc} \bar \eta\epsb& 0\\0 & 1\end{array}\right)$.
Here $\dim \mclt_p =4$, $\dim H^0(G_p,Ad\rhob)=2$
and the ordinary deformation ring is smooth in $4$ variables.
\newline\noindent
2) $\rhob \mid_{G_p}=\left(\begin{array}{cc} \bar \eta\epsb & *\\0 & 1\end{array}\right)$.
Here $\dim \mclt_p =3$, $\dim H^0(G_p,Ad\rhob)=1$
and the ordinary deformation ring is smooth in $3$ variables.
\newline\noindent
3) $\rhob \mid_{G_p}=\left(\begin{array}{cc} \epsb & *\\0 & 1\end{array}\right)$ is flat.
Here $\dim \mclt_p =3$, $\dim H^0(G_p,Ad\rhob)=1$
and the ordinary deformation ring is smooth in $3$ variables.
\newline\noindent
4) $\rhob \mid_{G_p}=\left(\begin{array}{cc} \epsb & *\\0 & 1\end{array}\right)$ is not flat.
Here $\dim \mclt_p =3$, $\dim H^0(G_p,Ad\rhob)=1$
and the ordinary deformation ring is smooth in $3$ variables.
\newline\noindent
5) $\rhob \mid_{G_p}=\left(\begin{array}{cc} \epsb & 0\\0 & 1\end{array}\right)$. Here 
$\dim \mclt_p=5$ and 
the ordinary deformation ring is not smooth, but it has a smooth quotient in four variables
whose characteristic zero points include all points of weight $k>2$ and all flat
points of weight $k=2$.
So we redefine $\mclt_p$ to be the $4$ dimensional subspace induced by this quotient and
note $\dim H^0(G_p,Ad\rhob)=2$.
\end{prop}
\begin{proof} Let $U \subset Ad\rhob$ be the upper triangular matrices. In each case we will compare
the unrestricted (local) upper triangular deformation theory of $\rhob \mid_{G_p}$ to its (local) 
ordinary deformation theory.

1) 
$\mclt_p$ includes the two unramified twists, the one ramified twist of $\bar \eta\epsb$ and the nontrivial
extension of $1$ by $\bar \eta\epsb$ so $\dim \mclt_p=4$. 
One easily sees that  $\dim H^2(G_p,U)=0$ and $\dim H^1(G_p,U)=5$ so the upper triangular deformation
ring is smooth in $5$ variables. Its ordinary quotient is formed by forcing
the lower right entry to be unramified. 
This involves the one relation that comes from setting the ramified part of the lower right entry, when evaluated
at a topological generator of the Galois group over $\QQ_p$ of the cyclotomic extension, to be trivial.
Since this relation necessarily cuts the tangent space down by one
variable, we can  take it to be a variable of the $5$ dimensional upper triangular ring,  so the ordinary ring is smooth in
$4$ variables.

2)
One computes
 $\dim H^2(G_p,U)=0$ and $\dim H^1(G_p,U)=4$ so the upper triangular
deformation ring is smooth in $4$ variables. 
 Let $U^1$ be the matrices of the form $\left(\begin{array}{cr} a& b\\0 & -a \end{array}\right)$. 
One computes $\dim H^1(G_p,U^1) =2$ and
from Table $3$  of \cite{R2} we have $\dim \mcl_p=1$.
As $\mcl_p \subset H^1(G_p,U^1)$ any
element of $H^1(G_p,U^1)$ not in $\mcl_p$ is
 ramified on both diagonals. 
A linear combination
of this class and the ramified twist will be trivial on the lower right entry and thus
ordinary. Of course the unramified twist is  in $\mclt_p$ so $\dim \mclt_p=3$.
That the ordinary ring is smooth
in $3$ variables follows from the argument in the proof of 1).

3) That $\dim \mclt_p=3$ follows from Proposition $13$ of \cite{R4}. 
One easily sees that  $\dim H^2(G_p,U)=0$ and $\dim H^1(G_p,U)=4$ so the upper triangular defomation
ring is smooth in $4$ variables. As before its ordinary quotient involves one relation that forces
the lower right entry to be unramified which again 
implies the ordinary ring is smooth in
$3$ variables.

4) 
Then the short exact sequence $0 \to U^1 \to U \to U/U^1 \to 0$ and routine Galois
cohomology computations give that $H^1(G_p,U^1) \to H^1(G_p,U)$ is an injection
of a $2$-dimensional space into a $4$-dimensional space. 
The cokernel is spanned by the images of the ramified and unramified twists.
There are two independent
extensions of $1$ by $\bar \eps$
so at least one dimension of $H^1(G_p,U^1)$ is ordinary. 

If all of $H^1(G_p,U^1) \subset H^1(G_p,U)$ is ordinary then, taking
into account the unramified twist, $\dim \mclt_p \geq 3$.
The only way $ \dim \mclt_p=4=\dim H^1(G_p,U)$ is if  the ramified twist belongs to $\mclt_p$,
which we know does not happen. Thus $\dim \mclt_p=3$ in this case.

If only one dimension of $H^1(G_p,U^1)$ is ordinary
(this is what actually happens, but proving it is messier than the weaker argument used here)
then the 
same proof as in 2) implies
$\dim \mclt_p=3$.

Since $\dim \mclt_p = 3$ in all cases and the upper triangular ring is smooth in
$4$ variables, the  ordinary ring is smooth in
$3$ variables.


5) This case is a bit more involved as the ordinary ring is not smooth. 
First note that $\mclt_p$ contains the two unramified twists, the ramified twist of $\epsb$ and
the {\it two} extensions of $1$ by $\epsb$ and so $\dim \mclt_p=5$. We will replace it by a $4$-dimensional
subspace.

Let $D$ be the maximal pro-$p$ abelian extension of $L$, the the composite of
$\ZZ_p$-unramified extension of $\QQ_p$ and
$  \QQ_p( \mu_{p^{\infty}})  $. 
Then in this case, any ordinary deformation of $\rhob$ has meta-abelian image
and factors through
$Gal(D/\QQ_p)$. By Kummer theory $D$ is generated by $p$-power roots
of elements of $L$. Let 
$K \subset D$ be the subfield generated by $p$-power  roots
of {\it units} of $L$. 
\vskip1em
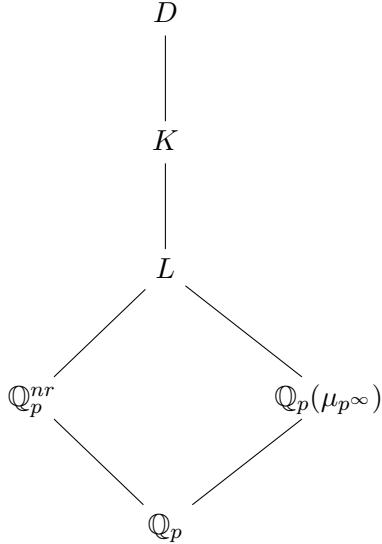
\hspace{5cm}
\begin{tikzpicture}[description/.style={fill=white,inner sep=2pt}]
\matrix (m) [matrix of math nodes, row sep=3em,
column sep=2.5em, text height=1.5ex, text depth=0.25ex]
{ & D & \\
& K & \\
& L
& \\
\QQ^{nr}_p & & \QQ_p(\mu_{p^{\infty}}) \\ 
& \QQ_p & \\};
\path (m-1-2) edge  (m-2-2);
\path (m-2-2) edge  (m-3-2);
\path(m-3-2) edge   (m-4-1);
\path(m-3-2) edge   (m-4-3);
\path(m-4-1) edge   (m-5-2);
\path(m-4-3) edge   (m-5-2);
\end{tikzpicture}
\newline\noindent
Let $\psi_1,\psi_2 :G_p \to \ZZ^*_p$ be unramified characters each congruent to $1$ mod $p$.

We will consider a series of ring homomorphisms
$$ R^{ord} \twoheadrightarrow R^{ord,unit}\twoheadrightarrow   R^{ord,unit}_k \twoheadrightarrow 
 R^{ord,unit,\psi_2}_k 
 \twoheadrightarrow R^{ord,unit,\psi_1,\psi_2}_k$$
where the superscript `unit'  indicates the quotient of the ordinary ring whose
deformation  factors
through $Gal(K/\QQ_p)$ and the presence of the unramified character $\psi_i$
as a superscript indicates that we are fixing $\psi_i$ in the $ii$ spot on the diagonal.
The subscript $k$ indicates the weight. 
So $R^{ord,unit,\psi_1,\psi_2}_k$
is the deformation ring parametrizing deformations of $\rhob \mid_{G_p}$ that factor through
$Gal(K/\QQ_p)$ and are of the form
$\left(\begin{array}{cc} \epsilon^{k-1}{\tilde{\epsilon}}^{2-k} \psi_1 & * \\ 0 & \psi_2 \end{array}\right)$.
For instance, the ring $R^{ord,unit}_k$ puts no restrictions on the unramified diagonal characters.

Consider $R^{ord,unit,\psi_1,\psi_2}_k$.
The tangent space for this ring is $1$-dimensional as follows. 
No twists by characters on the diagonal are allowed and
the 
 tr\`{e}s ramifi\'{e}e
extension of $1$ by $\bar \eps$ is not allowed either. 
Only the peu ramifi\'{e}e
extension of $1$ by $\bar \eps$ is allowed.
Thus the corresponding deformation ring is $\ZZ_p[[U]]/I_1$. If $I_1$ contains a nonzero element $g(U)$,
then by the Weierstrass preparation theorem we can assume $g(U)=p^rh(U)$ 
where $h(U)$ is a distinguished polynomial,
or $h(U) \equiv 1$ or  $0$. 
But Lemma~\ref{units} gives the existence of nonsplit characteristic zero deformations. 
Conjugating
these by $\left(\begin{array}{cc} p^m & 0\\0 & 1\end{array}\right)$ gives {\em different
deformations} of $\rhob$  for each $m$ (though of course these representations are all isomorphic), so our ring has infinitely many characteristic
zero points and $h(U)$ would have infinitely many roots, 
a contradiction. Thus $h(U) \equiv 0$, $I_1$ is trivial and 
$R^{ord,unit,\psi_1,\psi_2}_k \simeq \ZZ_p[[U]]$.

The ring $R^{ord,unit,\psi_2}_k$ has two dimensional tangent space (the unramified twist of $\bar \eps$ is now allowed)
and so $R^{ord,unit,\psi_2}_k \simeq \ZZ_p[[U_1,U_2]]/I_2$. But for each choice of $\psi_1$, we see this ring
has a different quotient isomorphic to $\ZZ_p[[U]]$.  If $I_2 \neq (0)$ 
Krull's principal ideal theorem (see Corollary $11.18$ of \cite{1}) implies
$R^{ord,unit,\psi_2}_k$ has Krull dimension at most $2$. Then
a Noetherian ring of Krull dimension at most
$2$ has infinitely many components of Krull dimension $2$, a contradiction. Thus $I_2=0$.
Similarly $R^{ord,unit,}_k$ has three dimensional tangent space 
and so $R^{ord,unit}_k \simeq \ZZ_p[[U_1,U_2,U_3]]/I_3$. But for each choice of $\psi_2$, we see this ring
has a different quotient isomorphic to $\ZZ_p[[U_1,U_2]]$.  If $I_3 \neq (0)$ the same Krull dimension
argument as above
gives a contradiction.
Thus $I_3=0$.
Finally, $R^{ord,unit}$ has four dimensional tangent space as only 
the tr\`{e}s ramifi\'{e}e extension of $1$ by $\bar \eps$ is not allowed.
We have, for each $k \geq 2$, $$R^{ord,unit}  \twoheadrightarrow R^{ord,unit}_k \simeq \ZZ_p[[U_1,U_2,U_3]]$$ 
 so, arguing as before,
$R^{ord,unit} \simeq \ZZ_p[[U_1,U_2,U_3,U_4]]$.

Using Lemma~\ref{units} we see 
it remains to show that weight $2$ flat deformations of $\rhob$ factor through $R^{ord,unit}$
in the $\psi_1=\psi_2=\psi$ case. 
This follows immediately from Lemma~\ref{fontaine}.

\end{proof}

\begin{prop} \label{lv} For  $v \neq p$, 
$\dim \mclt_v =\dim H^0(G_v,Ad\rhob)$ and
$\dim \mclt_p = \dim H^0(G_p,Ad\rhob)+2$
\end{prop}
\begin{proof} 
For $v\neq p$ it is known that $\dim \mcl_v = \dim H^0(G_v,\ad)$. 
As we switch from $\ad$ to $Ad\rhob$
$$\dim H^0(G_v,Ad\rhob) -\dim H^0(G_v,\ad) =1=\dim \mclt_v -\dim \mcl_v.$$ 

The $v=p$ result follows from Fact~\ref{vequalsp}.
\end{proof}

Propositions~\ref{wiles} and~\ref{lv} give, taking into account $v=\infty$, 
\begin{cor} \label{prop16} 
$\dim H^1_{\mclt}(G_S,Ad\rhob) - \dim H^1_{{\mclt}^{\perp}}(G_S,Ad\rhob^*)=1$.
\end{cor}

\subsection{Local deformation rings at nice primes}

Finally, we recall the notion of {\it nice} primes for a representation. 
The definition given below is a blend of that given in 
\cite{R3} (see  \S  $2$ and Propostion $2.2$) and that of \cite{T} that is suited for our purposes.
The latter reference handles the case $p=3$.

\begin{Def} \label{np}
Let $\rhob:\gal \to GL_2(\FF_{p^f})$ odd, ordinary, full and weight $2$ trivial nebentype
be  given. 
Let $R$ be a finite cardinality complete local Noetherian ring with
residue field $\FF_{p^f}$ and
let $\rho_R$ be a
continuous homomorphism
$\gal \to GL_2(R)$ lifting $\rhob$.
A prime $q \not \equiv  1 $ mod $p$ is called {\bf nice}
if $\rhob$ is unramified at $q$ and if the eigenvalues of $\rhob(Fr_q)$ are $q$ and $1$.
The prime $q$ is called {\bf $\rho_R$-nice} if $\rho_R$ is unramified at $q$, 
$\rho_R(Fr_q)$ has eigenvalues  $q$ and $1$ and order prime to $p$. 
\end{Def}
The local at $q$
deformation ring has a smooth quotient whose points
${\mathcal C}_q$ consist of Steinberg deformations. There is an
induced subspace $\mcl_q \subset H^1(G_q,\ad)$.
When $ q \not \equiv  \pm 1$ mod $p$, there is a single family of Steinberg deformations.
When $q \equiv -1$ mod $p$, a necessity in the $p=3$ case, there are two families
of Steinberg deformations. Taylor simply chooses a family and defines 
${\mathcal L}_q$ and ${\mathcal C}_q$ accordingly. See \cite{T} for all of this.


\begin{prop}\label{np2} Let 
 $\rho_R$ be odd, ordinary,  weight $2$ with full reduction $\rhob$. 
Then $\rho_R$-nice primes exist and for any nice prime $q$,
$\dim \mcl_q = \dim H^0(G_q,\ad)=1$ and a smooth quotient of the deformation
ring exists with points ${\mathcal C}_q$.
Proposition~\ref{lv} applies for nice primes so $\dim \mclt_q = 2 =\dim H^0(G_q,Ad\rhob)$
and ${\tilde{\mathcal C}}_q$ consists of all unramified twists of points of ${\mathcal C}_q$.
\end{prop}
\begin{proof}
We are given that $\rhob$ is full and det$ \rhob =\bar \eps$ so for $a \in \FF_p$, $a\neq 1$, choose
$\left(\begin{array}{cc} a & 0 \\0 &1\end{array}\right) \in \mbox{ image} (\rhob)$. 
Any prime $q$ with Frobenius in the conjugacy class of this matrix will be nice. 
After lifting this matrix to an element of image$(\rho_R)$ and raising to a large power of $p$ (say $p^r$),
the new matrix will have tame order and eigenvalues $\{a^{p^r},1\}$ which are distinct.
Any prime 
with Frobenius 
in the conjugacy class of this element will be $\rho_R$-nice.

The cohomological results are standard and we do not give their proofs.
\end{proof}

\section{Ordinary  smooth deformation rings and  $\rho_n$ }
For any finite set of primes $T \supset S$ with $T \backslash S$ consisting of only nice primes,
we  have an ordinary arbitrary weight deformation ring denoted $R^{ord,T-new}$ and its weight
$2$ quotient $R^{ord,T-new}_2$. When restricted to $G_v$  the points of these rings
lie, respectively,  in ${\tilde{\mathcal C}}_v$ and ${\mathcal C}_v$. We remark that in previous
papers we used the notation  $T\backslash S_0$-new to indicate that all nice primes
were in the level of the modular form. Since it is less cumbersome, we use 
$T$-new here instead. 
Results toward Theorem~\ref{easy} are proved in \cite{CP}.
\begin{theorem} \label{easy}
Suppose  $\rho_n:\gal \to GL_2(\mco/(\pi^n))$ is odd,  ordinary, weight $2$, modular, has full image and determinant $\eps$,  and
is  balanced.
Suppose $\rho_n \mid_{G_v} \in  {\mathcal C}_v$ for all $v \in S$.
Then there  exists a finite set of primes $T\supseteq S$ such that the universal ordinary
 `new at $T$' ring  $R^{ord,T-new} \simeq\wfq[[U]]$ and there are surjections
$$R^{ord,T-new} \twoheadrightarrow R^{ord,T-new}_2  \twoheadrightarrow \mco/(\pi^n)$$ 
from this ring to its weight $2$ quotient and then to $\mco/(\pi^n)$ inducing $\rho_n.$
\end{theorem}

It is {\bf not} a consequence of Theorem~\ref{easy} that $\rho_n$ lifts to 
a $T$-$new$ weight $2$ characteristic zero representation. 
For instance, if $n=3$ and $\mco=\ZZ_p[\sqrt{p}]$ it is possible that 
$$R^{ord,T-new}_2\simeq \ZZ_p[[U]]/\left((U-p)(U-2p)(U-3p)\right)$$ and
$\rho_{3}$ arises from $U\mapsto \sqrt{p}$.
The smoothness of 
$R^{ord,T-new}$
immediately implies the existence of characteristic zero lifts, but these lifts may not have
classical weight, let alone weight $2$. Theorem~\ref{hard}
addresses this.

\subsection{Group theoretic lemmas}

We need the following lemma of Boston (see \cite{BO}) and Lemma~\ref{p35} for Lemma~\ref{shekharslemma}.
\begin{lemma} \label{Boston} (Boston)
Let $p \geq 3$. Let $R$ be a complete local Noetherian ring with residue
characteristic $p$. Let $\rho:G \to GL_2(R)$ be a representation and assume
the image of the projection $$\rho_2:G \to GL_2(R) \to GL_2\left(R/m^2_R\right)$$ is full, that is it
 contains $SL_2\left(R/m^2_R\right)$. Then the image of $\rho$ contains $SL_2(R)$.
\end{lemma}

Lemma~\ref{p35} below is well-known, but as we could find no precise reference we include  a proof.
\begin{lemma}\label{p35}
Let $G \subset GL_2(\FF_{p^f})$ be a full subgroup.
If $\FF_{p^f} \neq \FF_5$ then $H^1(G,\ad)=0$. If $\FF_{p^f}=\FF_5$, 
suppose there exists 
$\left(\begin{array}{cc} r&0\\0 & s\end{array}\right) \in G$ with $\displaystyle \frac{r}s \neq \frac{s}r$.
Then $H^1(G,\ad)=0$.
\end{lemma}
\begin{proof} 
That $H^1(SL_2(\FF_3),\ad)=0$ is well-known. See for instance Lemma $1$ of \cite{R5}. Since
$ 3 \nmid [G:SL_2(\FF_3)]$ the restriction map $H^1(G,\ad) \to H^1(SL_2(\FF_3),\ad)$
is injective and the result follows for $\FF_3$.

Let $B$
be the upper triangular Borel subgroup of $G$ and let $N$ be its unipotent subgroup. 
Fullness implies $p \nmid [G:B]$ so again the restriction map
$H^1(G,\ad) \to H^1(B,\ad)$ is injective. We will prove $H^1(B,\ad)=0$.

Let $U^0 \subset U^1 \subset \ad$
where $U^0$ is the space of upper triangular nilpotent  matrices
and $U^1$ is the space of upper triangular trace zero matrices. From the short exact
sequence
$$0 \to U^1 \to \ad \to \ad/U^1 \to 0$$
we easily get
$$\dots \to 0 \to H^1(B,U^1) \to H^1(B,\ad) \to H^1(B,\ad/U^1) \to \dots$$
so it suffices to prove the (nonzero) outside terms are trivial.

Associated to $H^1(B,\ad/U^1)$ we the exact inflation-restriction sequence
$$0\to H^1(B/N, (\ad/U^1)^N) \to H^1(B, \ad/U^1) \to H^1(N,\ad/U^1)^{B/N} \to H^2(B/N,(\ad/U^1)^N).$$
As $p \nmid  |B/N|$ the outside terms are trivial. One easily sees $N$ acts trivially on $\ad/U^1$ 
so $H^1(N,\ad/U^1)^{B/N}=Hom_{B/N}(N,\ad/U^1)$.
Observe $\left(\begin{array}{cc} r&0\\0 & s\end{array}\right) \in B/N$ acts on $N$ by $\displaystyle \frac{r}s$ and on
the $1$-dimensional space $\ad/U^1$ by $\displaystyle \frac{s}r$. As long as there is an element as above
where $\displaystyle \frac{r}s \neq \frac{s}r$, we have $H^1(N,\ad/U^1)^{B/N}=0$ so $H^1(B,\ad/U^1)=0$. 
Choose $r=a$ and $s=a^{-1}$ so we are requiring $a^4 \neq 1$. Such $a$ exist in $\FF_{p^f}$ when 
$p^f\neq 3,5$
The extra hypothesis for $p^f=5$ gives that $H^1(B,\ad/U^1)=0$ in this last case.

For the $H^1(B,U^1)$ term, consider the short exact sequence
$$ 0\to U^0 \to U^1 \to U^1/U^0 \to 0$$ and take its $B$-cohomology to get
$$\dots 0 \to \FF_{p^f} \to H^1(B,U^0) \to H^1(B,U^1) \to H^1(B,U^1/U^0).$$
An easy analysis of the exact inflation-restriction sequences
$$0 \to H^1(B/N,(U^0)^{N})  \to H^1(B,U^0) \to H^1(N,U^0)^{B/N} \to
H^2(B/N, (U^0)^{N})$$
and
$$0 \to H^1(B/N, (U^1/U^0)^{N})  \to H^1(B,U^1/U^0) \to H^1(N,U^1/U^0)^{B/N} \to
H^2(B/N, (U^1/U^0)^{N})$$
gives that $\dim H^1(B,U^0)=1$ and $\dim H^1(B,U^1/U^0)=0$. Thus  $H^1(B,U^1)=0$
and the proof is complete.
\end{proof}
\begin{lemma} \label{shekharslemma}Let $p \geq 3$. Let $G \subset GL_2(\mco/(\pi^r))$ be a subgroup. Suppose the hypotheses
of Lemma~\ref{p35} are satisfied for the image of $G \to GL_2(\FF_{p^f})$
and that the
 image
of the projection $p_2:G \to GL_2(\mco/(\pi^2))$
is full.
Then $\dim H^1(G,\ad)  =1$. 
\end{lemma}
\begin{proof} 
Since the image of $p_2$ is full    the hypothesis of
Lemma~\ref{Boston} is satisfied so
$G \supset SL_2(\mco/(\pi^r))$. 

Let $\Gamma$ 
be the kernel of the projection $G \to GL_2(\FF_{p^f})$. 
We have the exact inflation-restriction sequence
$$0 \to H^1(G/\Gamma,\ad^{\Gamma}) \to H^1(G,\ad) \to 
H^1(\Gamma,\ad)^{G/\Gamma}.$$
As $G/\Gamma$ is the image of the projection $p_1:G \to GL_2(\FF_{p^f})$,
Lemma~\ref{p35} implies
the first term is trivial.

Also, as $\Gamma$ acts trivially on $\ad$,
$$H^1(\Gamma,\ad)^{G/\Gamma} =Hom_{G/\Gamma}(\Gamma,\ad).$$
For any $\gamma \in Hom_{G/\Gamma}(\Gamma,\ad)$,
Kernel$(\gamma) \supset \Gamma^{'}$, the commutator subgroup of $\Gamma$.
 
Set $a=1+\pi$ and $r = \displaystyle \frac{\pi x}{2+\pi}$ and use fullness to see
$$\left(\begin{array}{cc}  a & 0 \\ 0 & a^{-1} \end{array}\right)
\left(\begin{array}{cc}  1 & r \\ 0 & 1 \end{array}\right)
\left(\begin{array}{cc}  a & 0 \\ 0 & a^{-1} \end{array}\right)^{-1}
\left(\begin{array}{cc}  1 & r \\ 0 & 1 \end{array}\right)^{-1}=
\left(\begin{array}{cc}  1 & \pi^2 x \\ 0 & 1 \end{array}\right) \in \Gamma^{'} \subset \mbox{ Kernel}(\gamma).$$
Similarly,
$$\left(\begin{array}{cc}  1 & 0 \\ \pi^2 y & 1 \end{array}\right) \in \Gamma^{'} \subset \mbox{ Kernel}(\gamma) .$$
As Kernel$(\gamma)$ is stable under the action of $SD$,
$$\left(\begin{array}{cc}  1 & 0 \\ 1 & 1 \end{array}\right) 
\left(\begin{array}{cc}  1 & \pi^2 z \\ 0& 1 \end{array}\right) 
\left(\begin{array}{rc}  1 & 0 \\ -1 & 1 \end{array}\right) 
=\left(\begin{array}{cc}  1 -\pi^2z& \pi^2 z \\ -\pi^2 z & 1+\pi^2z \end{array}\right) 
\subset \mbox{ Kernel}(\gamma).$$
Multiplying on the  left and right by suitable matrices
$$\left(\begin{array}{cc}  1 & 0 \\ \pi^2 y & 1 \end{array}\right) \mbox{ and }
\left(\begin{array}{cc}  1 & \pi^2 x \\ 0 & 1 \end{array}\right) $$ we have
$$ \left(\begin{array}{cc}  1+\pi^2 z & 0 \\ 0 & (1+\pi^2 z)^{-1} 
\end{array}\right) \in \mbox{Kernel}(\gamma).$$
As every element of
$$\Gamma_2:=\{ A \in SL_2(\mco/(\pi^r)) \mid
A \equiv I \mbox{ mod } (\pi^2) \}$$ can be written as a product
$$A = \left(\begin{array}{cc}  1 & 0\\  \pi^2 y& 1 \end{array}\right) 
\left(\begin{array}{cc}  1+\pi^2 z & 0 \\ 0& (1+\pi^2z)^{-1} \end{array}\right) 
\left(\begin{array}{cc}  1 & \pi^2 x \\ 0& 1 \end{array}\right) $$
we have Kernel$(\gamma) \supset \Gamma_2$.
Since $\Gamma/\Gamma_2 \simeq \ad$,
$$\dim Hom_{G/\Gamma}(\Gamma,\ad) \leq 
\dim Hom_{G/\Gamma}(\Gamma/\Gamma_2,\ad) = \dim Hom_{G/\Gamma}(\ad,\ad) =1$$
so $\dim H^1(G,\ad) \leq 1$. As $\mco/\ZZ_p$ is ramified,
$GL_2\left(\mco/(\pi^2)\right)\simeq GL_2\left(\FF_{p^f}[U]/(U^2)\right)$ is nontrivial
as we are given full image so $\dim H^1(G,\ad)=1$.
\end{proof}

\subsection{Selmer goups}

All sets of primes $Z$  below will be finite, contain $S$ and 
$Z\backslash S$ will consist of nice primes.

\begin{prop} \label{sub}Let $h \in H^1_{\mclt}(G_Z,Ad\rhob)$ and 
$\phi \in H^1_{ {\mclt}^{\perp}}(G_S,Ad\rhob^*)$. If $h  \in H^1(G_Z,\FF_{p^f})
\subset H^1(G_Z,Ad\rhob)$ then $h=0$.
If $\phi  \in H^1(G_Z,\FF_{p^f}(1)) \subset H^1(G_Z,Ad\rhob^*)$ then $\phi=0$.
\end{prop}
\begin{proof} 
If $h \in H^1(G_Z,\FF_{p^f})$, it corresponds,
when viewed as a lift to the dual numbers, to a twist by a character which gives
rise to a $\ZZ/(p)$-extension of $\QQ$. Definition~\ref{tildes} implies that 
$\mclt_v \cap H^1(G_v,\FF_{p^f})$ 
is spanned, for all $v$, by the $\FF_{p^f}$-valued unramified
twists so the corresponding global extension is unramified everywhere so
$h=0$.

Set $M =\FF_{p^f}$ and for all $v$ set 
$${\mathcal M}_v = \mclt_v \cap H^1(G_v,\FF_{p^f}) = H^1_{nr}(G_v,\FF_{p^f}) .$$
We just showed $H^1_{\mathcal M}(G_Z,M)=0$. 
As $\dim {\mathcal M}_v =\dim H^0(G_v,M)=1$ for $v \neq \infty$,  
Proposition~\ref{wiles}   gives
$\dim H^1_{ {\mathcal M}^{\perp}}(G_Z,M^*)=0$ as well.

Any $\phi \in H^1(G_Z,\FF_{p^f}(1)) \cap H^1_{{\mclt}^{\perp}}(G_Z,Ad\rhob^*)$ 
cuts out an extension $L/\QQ(\mu_p)$ that is Galois over
$\QQ$ and $Gal(\QQ(\mu_p)/\QQ)$ acts on $Gal(L/\QQ(\mu_p))$ by $\bar \epsilon$.
At $v \neq p$ unramified cohomologies are exact annihilators under the local duality
pairing so ${\mclt}^{\perp}_v\cap H^1(G_v,\FF_{p^f}(1))
=H^1_{nr}(G_v,\FF_{p^f}(1))$.
(This last group is trivial unless $v\equiv 1$ mod $p$). So for $v \neq p$, $L/\QQ(\mu_p)$ is unramified
at $v$.

For $v=p$ choose a  subspace $V \subset H^1(G_p,Ad\rhob)$ such that
$$\mclt_p = 
\left(\mcl_p \oplus {\mathcal M}_p \right)+V,\,\,\,
V \cap \left(\mcl_p \oplus {\mathcal M}_p \right) = 0$$
so
$$\mclt^{\perp}_p = 
\left(\mcl_p \oplus {\mathcal M}_p \right)^{\perp} \cap V^{\perp}=
\left(\mcl^{\perp}_p \oplus {\mathcal M}^{\perp}_p \right) \cap V^{\perp}$$
and thus
$$\mclt^{\perp}_p \cap H^1(G_p,\FF_{p^f}(1)) \subset {\mathcal M}^{\perp}_p .$$
Thus 
$\phi \mid_{G_v} \in {\mathcal M}^{\perp}_v$ for all $v$, that is $\phi \in H^1_{ {\mathcal M}^{\perp}}(G_Z,M)$
which we already proved is trivial
so $\phi=0$.
\end{proof}


Note that for $h \in H^1(G_Z,\ad)$ and $q \notin Z$ nice, 
$h\mid_{G_q} \neq 0$ is equivalent to $h\mid_{G_q} \notin \mcl_q$. Similarly,
for $h \in H^1(G_Z,Ad\rhob)$ write $h= h_{\ad} +h_{sc}$ with $h_{\ad} \in H^1(G_Z,\ad)$
and $h_{sc} \in H^1(G_Z,\FF_{p^f})$
For $q \notin Z$ nice, 
$h\mid_{G_q} \notin \mclt_q$ is equivalent to $h_{\ad}\mid_{G_q} \notin \mcl_q$
which we just saw is equivalent to $h_{\ad}\mid_{G_q} \neq 0$.

\begin{prop} \label{nice}
Let $h\in H^1_{\mcl}(G_Z,\ad)$,
$\phi \in H^1_{{\mclt}^{\perp}}(G_Z,\ad^*)$ and let $q$ be nice.
\newline\noindent
1)  
The injective inflation 
map $$H^1(G_Z,\ad) \to H^1(G_{Z \cup \{q\}},\ad)$$ has codimension $0$ or $1$.
If  $\Sh^1(\ad^*) \mid_{G_q} = 0$ then the  codimension is $1$.
\newline\noindent
2) If  $\phi \mid_{G_q} \neq 0$ then
the maps 
$$H^1(G_{Z\cup \{q\}},\ad) \to \oplus_{v\in Z} \frac{H^1(G_v,\ad)}{{\mathcal L}_v}\mbox{ and }
H^1(G_Z,\ad) \to \oplus_{v\in Z} \frac{H^1(G_v,\ad)}{{\mathcal L}_v}$$ have the same kernel.
\newline\noindent
3) If $h,\phi\mid_{G_q} \neq 0$ then
$$\dim H^1_{\mcl}(G_{Z \cup \{q\}},\ad) = H^1_{\mcl}(G_Z,\ad) -1,\,\,
\dim H^1_{\mcl^{\perp}}(G_{Z \cup \{q\}},\ad^*) = \dim H^1_{\mcl^{\perp}}(G_Z,\ad^*) -1.$$
\newline\noindent
4) If $   H^1_{\mcl}(G_Z,\ad)  \mid_{G_q} =0$, $\phi \mid_{G_q} \neq 0$ then
$$H^1_{\mcl}(G_{Z \cup \{q\}},\ad) = H^1_{\mcl}(G_Z,\ad),\,\,
\dim H^1_{\mcl^{\perp}}(G_{Z \cup \{q\}},\ad^*) = \dim H^1_{\mcl^{\perp}}(G_Z,\ad^*).$$
\newline\noindent
5) If $   H^1(G_Z,\ad^*) 
\mid_{G_q} = 0$ then
$$H^1(G_{Z \cup \{q\}},\ad)  \to \oplus_{v\in Z} \frac{H^1(G_v,\ad)}{\mcl_v}$$
and 
$$H^1(G_Z ,\ad)  \to \oplus_{v\in Z} \frac{H^1(G_v,\ad)}{\mcl_v}$$
have the same image.
\end{prop}
\begin{proof} As the proofs of all parts are similar, 
we only prove part 2).  We use the normal local Selmer
condition for $v\in Z$, but just for this proof we set $\mcl_q=H^1(G_q,\ad)$
so $\mcl^{\perp}_q=0$.
We apply Proposition~\ref{wiles} with the sets $Z$ and $Z\cup \{q\}$. Then
$$ \dim H^1_{\mcl}(G_{ Z\cup \{q\}},\ad) - \dim H^1_{\mcl^{\perp}}(G_{Z\cup \{q\}},\ad^*)
=\dim H^1_{\mcl}(G_Z,\ad) - \dim H^1_{\mcl^{\perp}}(G_Z,\ad^*)+1.$$
As $\mcl^{\perp}_q=0$ we have 
$H^1_{\mcl^{\perp}}(G_{Z\cup \{q\}},\ad^*) \subset H^1_{\mcl^{\perp}}(G_Z,\ad^*)$ and since 
$\phi \mid_{G_q} \neq 0$ this containment is proper. As dual Selmer goes down by $1$ in dimension
as we switch from $Z$ to $Z\cup \{q\}$ the above equation implies the dimension of Selmer
does not change as we switch from $Z$ to $Z\cup \{q\}$. Since $\mcl_q=H^1(G_q,\ad)$ we have
$H^1_{\mcl}(G_Z,\ad) \subset H^1_{\mcl}(G_{ Z\cup \{q\}},\ad)$ and the result follows.
\end{proof}

\begin{prop} \label{nicefull}
Let $h\in H^1_{\mcl}(G_Z,Ad\rhob)$,
$\phi \in H^1_{{\mclt}^{\perp}}(G_Z,Ad\rhob^*)$ and $q$ be nice.
\newline\noindent
1) The injective inflation 
map $$H^1(G_Z,Ad\rhob) \to H^1(G_{Z \cup \{q\}},Ad\rhob)$$ has codimension $0$ or $1$.
If  $\Sh^1(Ad\rhob^*) \mid_{G_q} = 0$ then the 
codimension is $1$.
\newline\noindent
2)   If $h\mid_{G_q} \notin \mclt_q\mbox{ and }\phi\mid_{G_q} \neq 0$ then
$$\dim H^1_{\mcl}(G_{Z \cup \{q\}},Ad\rhob) = H^1_{\mcl}(G_Z,Ad\rhob) -1,\,\,
\dim H^1_{\mcl^{\perp}}(G_{Z \cup \{q\}},Ad\rhob^*) = \dim H^1_{\mcl^{\perp}}(G_Z,Ad\rhob^*) -1.$$
\end{prop}
The proof of Proposition~\ref{nicefull} is similar to that of Propostion~\ref{nice} and not included.
See \cite{T} for the proof of 2).
\vskip1em

Consider the deformation
to the dual numbers given by
$$\gal \stackrel{\rho_n}{\to}GL_2\left(\mco/(\pi^n)\right) \to 
GL_2\left(\mco/(\pi^2)\right) \simeq GL_2(\FF_{p^f}[U]/(U^2)).$$
The fullness assumption implies $f\neq 0$.  As
the determinant of the above composite representation is $\bar \epsilon$,
$f \in H^1_{\mcl}(G_S,\ad) \subset H^1_{\mclt}(G_S,Ad\rhob)$, that is $f$ lives in the trace zero
cohomology. This is important as in the end $f$ will  span the tangent 
space of our arbitrary weight ordinary ring {\it and} 
the tangent space of its weight $2$ quotient. 
By Corollary~\ref{prop16} we may take $\{\phi_1,\dots,\phi_s\}$ 
and $\{h_1,\dots,h_s,f\}$ 
as bases of $H^1_{ {\mclt}^{\perp}}(G_S,Ad\rhob^*)$
and $H^1_{\mclt}(G_S,Ad\rhob)$.

\begin{lemma}\label{chebs} Let $Q_i$ be the set of nice primes such that, for $q_i \in Q_i$
\begin{itemize}
\item $\phi_i|_{G_{q_i}} \neq 0$ (equivalently, ${\phi}_{i,\ad^*} \mid_{G_{q_i}} \neq 0$),
\item 
         ${h}_{i,\ad} \mid_{G_{q_i}} \neq 0$ and
         ${h}_{i,sc} \mid_{G_{q_i}} =0$,
\item for $j \neq i$, $\phi_j,h_j|_{G_{q_i}} = 0$ and
\item $q_i$ is $\rho_n$-nice, that is $\rho_n(Fr_{q_i}) =\left(\begin{array}{cc} q_i & 0\\0 & 1\end{array}\right)$ where this element
has order prime to $p$.
\end{itemize}
Then $Q_i$ is nonempty.
\end{lemma}
\begin{proof} 
It suffices to show the  conditions above are independent Chebotarev conditions,
that is the determine linearly disjoint extensions over $K:=\QQ(\rhob)$,
 and thus
can be simultaneously satisfied.

Each of the cohomology classes above, when restricted to the absolute Galois group of the field
$K$, becomes an element of $Hom_{Gal(K/\QQ)}(G_K,M)$ for $M=Ad\rhob$
or $Ad\rhob^*$. For $M=\ad$ or $\ad^*$, the 
independence of the first three conditions has been established in \cite{R2} and \cite{T}. 
The case of full adjoint cohomology results
from these works and Proposition~\ref{sub} as follows.
Write  $\phi_i={\phi}_{i,\ad^*}+{\phi}_{i,\FF_{p^f}(1)}$ where ${\phi}_{i,\ad^*} \in H^1(G_S,\ad^*)$  and 
${\phi}_{i,\FF_{p^f}(1)} \in H^1(G_S,\FF_{p^f}(1))$. We claim the set 
$\{{\phi}_{1,\ad^*},\dots,{\phi}_{s,\ad^*}\}$ is independent.
Indeed, suppose $\displaystyle\sum^s_{j=1} a_j   {\phi}_{j,\ad^*}   =0$ is a dependence relation. Then 
$$\sum^s_{j=1} a_j \phi_j 
=\sum^s_{j=1} a_j(  {\phi}_{j,\ad^*}+{\phi}_{j,\FF_{p^f}(1)} ) = \sum^s_{j=1} a_j {\phi}_{j,\FF_{p^f}(1)}
\in H^1_{\mclt^{\perp}}(G_S,\ad^*) \cap 
H^1(G_S,\mu_p)$$ which is $0$ by Proposition~\ref{sub}, a contradiction. 
Let $L$ be the composite of the fields fixed by the kernels of $\phi_i \mid_{G_K}$. 
Then $Gal(L/K)$ contains, when viewed as a $\FF_{p^f}[Gal(K/\QQ)]$-module,
$s$ copies of $\ad^*$
by \cite{R1}, \cite{T}.
A similar argument gives that the composite of the fields fixed by the kernels of $h_i \mid_{G_K}$
and $f\mid_{G_K}$ contains $s+1$ copies of $\ad$. This reduces the independence of the first three
conditions to the same question with $\ad$ and $\ad^*$ cohomology where it is known.

The fourth condition is a complete splitting condition from $K$ to 
$L_n$,  the field fixed by the kernel of $\rho_n$.
The Jordan-H\"{o}lder components of 
$Gal(K_n/K)$ are  $\FF_{p^f}[Gal(K/\QQ)]$-submodules
 of $Ad\rhob$ that are either $Ad\rhob$ or $\ad$.
As the fields fixed by the kernels of the $\phi_i\mid_{G_K}$ give
$Ad\rhob^*$ (or $\ad^*$) extensions, these are linearly disjoint over $K$ from $L_n$.
The fields fixed by the  kernels of the $h_i \mid_{G_K}$ give rise to
$Ad\rhob$ (or $\ad$) extensions of $K$. If the composite of these fields intersects $L_n$ 
nontrivially then, as this intersection is abelian over $K$, the proof of
Lemma~\ref{shekharslemma} applied to $\rho_n$ implies this composite contains 
Kernel$(f\mid_{G_K})$
and 
$f$ is in the span of the trace zero parts of $\{{h}_{1,\ad},\dots,{h}_{s,\ad}\}$. 
Proposition~\ref{sub} then implies
$f$ is  in the span of $\{h_1,\dots,h_s\}$,  a contradiction.
Thus the composite of the fields fixed by the $h_i$ is linearly disjoint over $K$ from $L_n$.
\end{proof}

\subsection{Proof of Theorem \ref{easy}}

Lundell has proved the following
\begin{theorem}\label{lundell}  (Lundell) 
Let $\rhob$ be odd, ordinary, full  and weight $2$ with determinant $\eps$.
Suppose for a set $T$ $\dim H^1_{\mclt^{\perp}}(G_T,Ad\rhob^*)=0$. Then
$\dim H^1_{\mclt}(G_T,Ad\rhob)=1$ and
the ordinary arbitrary weight deformation ring $R^{ord,T-new} \simeq \wfq[[U]]
\simeq \TT^{ord,T-new}$, the universal ordinary modular deformation ring associated to $\rhob$.
Let $\{g\}$ be a basis of $H^1_{\mclt}(G_T,Ad\rhob)$. 
The weight $2$ quotient $R^{ord,T-new} _2 \simeq \wfq$ if and only if
$g \notin H^1_{\mcl}(G_T,\ad)$, that is if and only if $g$ does not belong to trace
$0$ cohomology.
\end{theorem}

\noindent{\bf Remark:} Using the ordinary deformation rings of Proposition~\ref{vequalsp},
this theorem handles all possible $\rhob \mid_{G_p}$. It does {\em not} handle
all possible $\rho_n \mid_{G_p}$ as $5$) of the proposition only allows
flat weight $2$ deformations.

\vskip1em
\noindent
{\it Proof of Theorem~\ref{easy}.}
Choose $q_i \in Q_i$ and set $Q =\{q_1,\dots ,q_s\}$ and $T=S\cup Q$. 
Part 2) of Proposition~\ref{nicefull} implies the Selmer and dual Selmer groups
decrease in dimension by $1$ for each $q_i$ at which we allow ramification. Thus
 $H^1_{\mclt^{\perp}}(G_T,Ad\rhob^*)=0$ and
$H^1_{\mclt}(G_T,Ad\rhob)$
 is spanned by $f$ so $R^{ord,T-new} \simeq\wfq_p[[U]]$. 
Since we assume $\rhob$ is modular and 
absolutely irreducible, \cite{DT} implies that
$R^{ord,T-new}$ has   characteristic zero points in all classical weights. Thus it
is in fact a Hida family.
By the fourth condition on the $q_i \in Q_i$ we see that that $\rho_n$ arises as a point of 
$R^{ord,T-new}_2$,
the weight $2$ quotient of $R^{ord,T-new}$. 
\hfill$\square$

\section{$\rho_n$ lifts to an $\mco$-valued weight $2$ point}

Suppose  $\rho_n:\gal \to GL_2(\mco/(\pi^n))$ is odd,  ordinary, weight $2$, modular, has full image,
with determinant  $\eps$. We assume $\rho_n$ is   balanced. Throughout this section we assume for simplicity that $\mco$ is totally ramified.

Suppose also that $\rho_n \mid_{G_v} \in  {\mathcal C}_v$ for all $v \in S$.
Recall that for $v=p$ we require that if $\rhob = 
\left(\begin{array}{cc} \epsb & 0\\0 & 1\end{array}\right)$ then ${\mathcal C}_p$
is taken to be the flat deformations, so we {\em must} assume $\rho_n \mid_{G_p}$
is flat. 

If $\rhob \mid_{G_p}=\left(\begin{array}{cc} \epsb & *\\0 & 1\end{array}\right)$
is flat and $\rho_n$ is flat (but not semistable), 
we take ${\mathcal C}_p$ to be the flat deformations.
If $\rho_n$ is not flat but semistable, we take ${\mathcal C}_p$ to be the semistable
deformations.  If $\rho_n$ is both flat and  semistable, we take ${\mathcal C}_p$
to be the flat or semistable deformations however we please. The universal representations corresponding to the  deformation rings considered in  Theorem \ref{hard}  are locally at $p$ of type ${\mathcal C}_p$.

The goal of this section is to prove Theorem~\ref{hard}.

\begin{theorem} \label{hard}

There exists a finite set of primes $T_2\supseteq S$ such that $R^{ord,T_2-new} \simeq\wfq[[U]]$
and  there are maps
$$\wfq[[U]] \simeq R^{ord,T_2-new} \twoheadrightarrow 
R^{ord,T_2-new}_2 \twoheadrightarrow \mco \twoheadrightarrow \mco/(\pi^n)$$
 inducing $\rho_n$. The map $R^{ord,T_2-new}_2 \twoheadrightarrow \mco$
is an isomorphism for $\rhob$ in cases $1)$, $2)$ and $4)$ of Proposition~\ref{vequalsp}. 
For the exceptional $\rhob$ described above
(cases $3)$ and $5)$ of Proposition~\ref{vequalsp}), we have that $\mco$ is isomorphic to either
the semistable or flat  quotient
of the weight $2$ deformation ring,  according
to how we have chosen ${\mathcal C}_p$.
\end{theorem}

In this section we will study various universal ordinary rings $R^{ord,?-new}$ associated
to $\rhob$ and their weight $2$ quotients $R^{ord,?-new}_2$. If we are in an exceptional
case of Theorem~\ref{hard} we need a quotient of the weight $2$ ring having property
$* \in \{flat,semistable\}$.
We will use the notation
$R^{ord,?-new}_{2^*}$ 
to indicate this quotient of the full weight $2$ ring. In cases $1$), $2$) and $4$) of Proposition~\ref{vequalsp} $R^{ord,?-new}_{2} \simeq R^{ord,?-new}_{2^*}$.


The
weight $2$ quotient of any $R^{ord,X-new}\simeq \wfq[[U]]$ is formed by quotienting out
$R^{ord,X-new}$
by a relation $w_2(U)$ that fixes the determinant of the $X$-new deformation to be the cyclotomic character.
\begin{lemma}\label{disting} The modularity of $\rhob$ implies
we may assume $w_2(U)\in \wfq[[U]]$ is a distinguished polynomial.
\end{lemma}
\begin{proof} Set $j(U)$
to be the determinant of our ordinary representation
evaluated at a topological generator of the Galois group of the cyclotomic $\ZZ_p$-extension of $\QQ$
so that at a weight $t$ point $U_t \in {\sf m}_{{\bar \ZZ}_p}$, $j(U_t)=(1+p)^{t-1}$.
Writing $j(U)=1+p+w_2(U)$,   the weight $2$ quotient is $\wfq[[U]]/(w_2(U))$.
The Weierstrass preparation theorem
implies $w_2(U) = p^tv_2(U)u(U)$ where $v_2$ is a distinguished polynomial and $u(U)$ is a 
unit. 
We need to prove  $t=0$ so suppose $t\geq 1$. 

By  \cite{DT}
$R^{ord,X-new}\simeq \wfq[[U]]$ has a point of each classical weight $k\geq 2$.
The weight $3$ quotient
is formed imposing the relation $j(U)=(1+p)^2$, that is quotienting out by 
$$j(U)-(1+p)^2=w_2(U) -p-p^2 = p^t v_2(U)u(U) -p-p^2 =
p\left[-1-p+p^{t-1} v_2(U)u(U)\right].$$
As there is at least one weight $2$ point in this Hida family, $v_2(U)$ has positive degree so
the rightmost quantity above is  $p$ times a unit and is never $0$ for any choice of $U$. Thus if $t\geq 1$ there
are no weight $3$ points, a contradiction. 
\end{proof}

The roots of $w_2(U)$ will give rise to all weight $2$ deformations of $\rhob$. In cases
$3$) and $5$) of Proposition~\ref{vequalsp}, we have $w_{2,fl}(U), w_{2,ss}(U) | w_2(U)$
where the roots of $w_{2,*}$ give rise to weight $2$ points with the appropriate local at $p$
property. {\it A priori}   $w_{2,fl}$ and $w_{2,ss}$ could share a root, but geometricity and the Weil bounds imply
this is not the case, though we do not need this for our purposes.
 Henceforth we will write $R/(w_{2,*}(U))$ to indicate the weight two quotient
with which we are dealing (the full weight $2$ quotient except in cases $3$) and $5$) of Proposition~\ref{vequalsp}) and will control to be $\mco$.

\subsection{Recollection  of earlier work in \cite{R3}}

Let $T$ be as in Theorem~\ref{easy}.
A key technical ingredient in this section is the main lifting result of \cite{R3}, which
in turn builds on  \cite{KLR}. The point of \cite{R3} was to build a pathological
Galois representation by removing all obstructions to deformation problems.
Here we repeat this procedure for a {\it finite number of steps}, but then we {\it introduce} an obstruction later to
force $R^{ord,T_2-new}_{2,*}$ to be `close to' a specified ring. This closeness will allow
us to choose $R^{ord,T_2-new}_{2,*}$ to be isomorphic to a given totally
ramified extension of $\wfq$.

We recall some of the key ingredients of \cite{R3}.
In \cite{R3} the only the fields $\FF_p$ for $p \geq 5$ were used. 
Recall that here we consider $\FF_{p^f}$ with $q=p^f$ and $p \geq 3$.
First consider the  hypotheses of section $4$ of \cite{R3}:

\begin{itemize}
\item  Fullness of the image of $\rhob$ we assume here. 
\item Triviality of $$\Sh^1_T(\ad^*) := \mbox{Kernel} \left( H^1(G_T,\ad^*) \to \oplus_{v\in T} H^1(G_v,\ad^*)\right)$$
can be  realised as follows. Note that for $T\subset Z$, $\Sh^1_Z(\ad^*) \subset \Sh^1_T(\ad^*)$. Then
let $\{\theta_1,\dots,\theta_t\}$ be a basis for $\Sh^1_T(\ad^*)$ and choose nice primes $q_i$
such that 
\begin{itemize}
\item $q_i$ is $\rho_n$-nice. Recall this is a complete splitting condition on $q_i$ in
$Gal(L_n/\QQ(\rhob,\mu_p))$,
\item $H^1(G_T,Ad\rhob) \mid_{G_{q_i}} =0$ and
\item  $\theta_i \mid_{G_{q_i}} \neq 0$ and  $j \neq i \implies\theta_j  \mid_{G_{q_i}} =0$.
\end{itemize}
Replacing   $T$ by $T\cup \{q_1,\dots,q_t\}$ (which we rename $T$) gives $\Sh^1_T(\ad^*)=0$.
\item  The third hypothesis of section $4$ of \cite{R3} was that the local deformations be specified
uniquely. This was equivalent to specifying $\wfq$ as our  smooth quotient of each local deformation ring.
We simply ignore that here and use ${\mathcal C}_v$ as our (weight $2$) set of local points as usual.
\end{itemize}

Suppose now we have an ordinary weight $2$ deformation of $\rhob$, $\rho_R:G_Z \to GL_2(R)$ where $R$
is a finite complete local Noetherian ring with residue field $\FF_{p^f}$ and $\rho_R \mid_{G_v} \in {\mathcal C}_v$
for all $v \in Z$. Let $R_1$ be another such
ring and $$ R_2 \twoheadrightarrow R_1 \stackrel{\delta}{\twoheadrightarrow} R$$ be  surjections with 
Ker$(\delta) = (b)$, a principal ideal 
isomophic
to $ \FF_{p^f}$.
It is natural to ask whether $\rho_R$ deforms to a 
$\rho_{R_1}:G_Z \to GL_2(R_1)$ of weight $2$.
The obstruction lies in $H^2(G_Z,\ad)$.
As $\Sh^1_T(\ad^*) $ is dual to $\Sh^2_T(\ad) $, the second  bullet
point above implies this obstruction is realised locally. But as  $\rho_R \mid_{G_v} \in {\mathcal C}_v$
for all $v \in Z$ and the ${\mathcal C}_v$ represent the points of a smooth ring, there are no local obstructions
and $\rho_{R_1}$ exists. It may be, however, that there are $v_0 \in Z$ with $\rho_{R_1} \mid_{G_{v_0}} \notin {\mathcal C}_{v_0}$.
In this case deforming to $R_2$ may not be possible.
The smoothness of the local deformation rings implies that $\rho_R \mid_{G_v}$ has a deformation to $GL_2(R_1)$
arising from ${\mathcal C}_v$ for all $v \in Z$. 
The obstruction to deforming $\rho_{R_1} \mid_{G_{v_0}}$ to $R_2$ can be removed by a cohomology class 
$z_{v_0}\in H^1(G_{v_0},\ad)$. We call the collection $(z_v)_{v \in Z}$ the {\it local condition problem} for $\rho_{R_1}$.
Proposition $3.4$ of \cite{R3} shows that there exists a $\rho_{R_1}$-nice prime $q$, and  $h \in H^1(G_{Z \cup \{q\}},\ad)$  that solves
the local condition problem above, that is $(I+ h)\rho_{R_1} \mid_{G_v}\in {\mathcal C}_v$ for all $v \in Z$. 
The difficulty is that we cannot guarantee $(I+ h)\rho_{R_1} \mid_{G_q}\in {\mathcal C}_q$. 
If this fails for all $\rho_{R_1}$-nice primes, Proposition $3.6$ of \cite{R3} shows
shows how to add two nice primes $q_1$ and $q_2$ to $Z$ and find a cohomology
class $h \in H^1(G_{Z \cup \{q_1,q_2\}},\ad)$ such that 
$(I+ h)\rho_{R_1} \mid_{G_v}\in {\mathcal C}_v$ for all $v \in Z \cup \{q_1,q_2\}$.



\subsection{Strategy of the proof of Theorem \ref{hard}}

We use the integer $N$ throughout this section to denote a large natural number.
This largeness will depend only on
$\rho_n:\gal \to GL_2(\mco/(\pi^n))$. We explain the strategy, which gives an indication of how 
$N$ is chosen.
The first step  is to construct a deformation problem where the
arbitrary
weight ordinary deformation ring will be $\wfq[[U]]$ and its weight $2$ quotient
will surject onto $\wfq[[U]]/(p^N,U^{Ne})$ which in turn will surject onto $\mco/(\pi^n)$
and give rise to $\rho_n$. So
$$\wfq[[U]] \simeq R^{ord} \twoheadrightarrow R^{ord}_{2^*} =\wfq[[U]]/(w_{2^*,N}(U))
\twoheadrightarrow   \wfq[[U]]/(p^N,U^{Ne})  \twoheadrightarrow      \mco/(\pi^n).$$
If in this composite $U$ maps to an element of $\mco/(\pi^n)$ whose various lifts
to $\mco$ have valuation greater than $1/e$ then the deformation to $\mco/(\pi^2)$
would be trivial, contradicting the fullness of $\rho_n$. Thus
$U$  maps to an element whose lifts to $\mco$ are uniformisers.
After multiplying by a unit, we may assume $U \mapsto \pi$.
Our strategy is to then alter the problem by allowing more ramification
so that $w_{2^*,N}(U)$ is exactly of degree $e$ and `close to' $g_{\pi}(U)$, the minimal
polynomial of $\pi$ over $W(\FF_{p^f})$. 
The choice of $N$ will depend on $|\pi^n|$ and the Krasner bound on the distances between roots
of $g_{\pi}(U)$.
Furthermore $w_{2^*,N}$ has a root
and has a root $y_{N,1}$
such that the deformation given by $U\mapsto y_{N,1}$ gives rise to $\rho_n$. Thus
$\rho_n$ will have a weight $2$ characteristic $0$ lifting.


\subsection{Weight 2 deformation rings that are large}
In this section we will construct large weight $2$ deformation rings that give rise
to $\rho_n$.  The technical hypotheses on $\rho_n \mid_{G_p}$
in the introduction arise here.

\begin{prop} \label{vli} For any integer $N$,
there exists a set $X_N \supseteq T$ such that  
$$R^{ord,X_N-new}_{2^*} 
\twoheadrightarrow   \wfq[[U]]/(p^N,U^{Ne})\twoheadrightarrow 
\mco/(\pi^n).$$
\end{prop}
\begin{proof} Theorem $1.1$ of  \cite{R3}, based on techniques of \cite{KLR},
gives examples of weight $2$ deformation rings that are arbitrarily large 
and ramified at infinitely many primes. It is proved by taking an inverse limit
of certain finite cardinality quotients of
deformation rings that satisfy a specified property at $v\in S_0$, namely the local representation at $G_v$ is 
(the reduction of) a specific deformation of $\rhob \mid_{G_v}$ to $\ZZ_p$.
While the ring
$\wfq[[U]]/(p^N,U^{Ne})$ is not explicitly included there, the techniques 
apply. The
deformation of Theorem~\ref{easy} to $\mco/(\pi^n)$ factors 
\begin{equation}\label{eq:surjs}
\wfq[[U]] = R^{ord,T-new}
\twoheadrightarrow 
R^{ord,T-new}_{2^*}=\wfq[[U]]/(w_{2^*}(U)) \twoheadrightarrow 
\mco/(\pi^n)
\end{equation}
and $U \mapsto \pi$ in the composite.
Let $g_{\pi}(U)$ be the minimal polynomial of $\pi$ over $\wfq$.
Observe that both $g_{\pi}(U)$ and $U^n$ are in the kernel of the composite map.
As $U \mapsto \pi$ gives an isomorphism $\wfq[[U]]/(g_{\pi}(U),U^n) \simeq \mco/(\pi^n)$
the kernel of ~\eqref{eq:surjs} is $(g_{\pi}(U),U^n)$ so $w_{2^*}(U) \in (g_{\pi}(U),U^n)$.
For $N \geq n$ note  $p^N, U^{Ne} \in (g_{\pi}(U),U^n)$ as they both  map to $0$ in ~\eqref{eq:surjs}.

We will construct a ring $R=\wfq[[U]]/I$ (not yet a deformation ring!) that surjects onto
$\wfq[[U]]/(p^N,U^{Ne})$ and from there onto  $\wfq[[U]]/(g_{\pi}(U),U^n) \simeq \mco/(\pi^n)$. 
We'll build $R$ as a series of small extensions of 
of $\wfq[[U]]/(g_{\pi}(U),U^n)$
and then invoke Proposition $3.6$ of  \cite{R3}
to realise all of these rings as quotients of  weight $2$ deformation rings.
Consider the map 
$$\wfq[[U]]/( pg_{\pi}(U),Ug_{\pi}(U),U^n) \twoheadrightarrow
\wfq[[U]]/(g_{\pi}(U),U^n).$$
The kernel is just the ideal  $(g_{\pi}(U))$ and this is
killed by $(p,U)$, the maximal ideal of $\wfq[[U]]$,
so the extension is small. Similarly,
the map 
$$\wfq[[U]]/( pg_{\pi}(U),Ug_{\pi}(U),pU^n,U^{n+1}) \twoheadrightarrow
\wfq[[U]]/( pg_{\pi}(U),Ug_{\pi}(U),U^n) $$
has kernel $(U^n)$ and this is also killed by $(p,U)$. Repeat this process (with more and more
elements in our ideal) until all the generators are of the form $p^rU^sg_{\pi}(U)$ or $p^rU^sU^n$ where
$r+s =N+Ne$. Let $I$ be this ideal of relations. In each relation
either $r\geq N$ or $s \geq Ne$ so $I \subset (p^N,U^{Ne}) \subset( g_{\pi}(U),U^n)$.
 Then
$$\wfq[[U]]/I \twoheadrightarrow \cdots \twoheadrightarrow
\wfq[[U]]/(g_{\pi}(U),U^n) \simeq \mco/(\pi^n)$$ 
is a series of small extensions.

Now we'll use 
\cite{R3} to  deform $\rho_n$  to each small extension, perhaps allowing ramification
at one or two nice primes at each step. If we can deform $\rho_n$ all the way to $\wfq[[U]]/I$
without allowing more ramification, then we are done. If not, there is a first place
at which the small deformation problem is obstructed. This is not at $\rho_n$
as $\Sh^2_T(\ad)=0$ (being dual to $\Sh^1_T(\ad^*)$) and the local defomation problems $\rho_n$
are assumed unobstructed.
The smoothness of the chosen quotients of the local deformation rings implies there are local
cohomology classes $(h_v)_{v\in T}$ that `unobstruct' each of the given local deformation
problems. Proposition $3.4$ of \cite{R3} implies that with one nice prime $q$ the local
deformation problems at $v\in T$ can be `unobstructed' by a global
class in $H^1(G_{T\cup \{q\}},\ad)$, but possibly this class introduces an obstruction
 at $q$. If all nice primes introduce such an obstruction, the rest of section $3$
of \cite{R3} shows how to allow ramification at two nice primes $\{q_1,q_2\}$ so that the
obstruction introduced at these primes cancel one another. Then we deform and 
move on to the next small extension.
Set $X_N$  to be the final set of nice primes.
\end{proof}

The primes $q$ used in Proposition~\ref{vli} 
were $\rho_n$-nice so
$f \mid_{G_q} =0$ and
 $f \in  H^1_{\mcl}(G_{X_N},\ad)\subset   H^1_{\mclt}(G_{X_N},Ad\rhob)$ 
but this last space could have  dimension $>1$
so the first step in our strategy is not yet complete.  Lemma~\ref{YN}
remedies this.

\begin{lemma} \label{YN}
There exists a set $Y_N$ containing $X_N$ of Proposition~\ref{vli} such that
$\dim H^1_{\mclt^{\perp}}(G_{Y_N},Ad\rhob^*)=0$ and 
$\dim H^1_{\mclt}(G_{Y_N},Ad\rhob)=1$ so
$R^{ord,Y_N-new} \simeq \wfq[[U]]$. Furthermore  
\begin{equation}\label{eq:seq}
\wfq[[U]] \simeq R^{ord,Y_N-new} \twoheadrightarrow
R^{ord,Y_N-new}_{2^*} =\wfq[[U]]/(w_{2^*,N}(U))\twoheadrightarrow
\wfq[[U]]/(p^N,U^{Ne})\twoheadrightarrow \mco/(\pi^n).
\end{equation}
\end{lemma}
\begin{proof} 
The proof is  similar to that of Lemma~\ref{chebs}.
Let $\rho_N:G_{X_N} \to GL_2\left( \wfq[[U]]/(p^N,U^{Ne})\right)$ 
be the deformation
of Proposition~\ref{vli}.
We will need that the image of $\rho_N$ is full. While this is easy to prove for $p>3$
as then ~\eqref{eq:nonsplit} is nonsplit, the proof below works for $p=3$
and exploits ordinariness.

The ring $\wfq[[U]]/(p^N,U^{Ne})$ has the quotient $\wfq/(p^2)$
and there is the  extension 
\begin{equation}\label{eq:nonsplit}
1\to \ad \to GL_2\left(\wfq/(p^2)\right) \to GL_2(\FF_{p^f}) \to 1. 
\end{equation}
When restricted to $G_p$
the deformation
$$\gamma:G_{Y_N}  \to GL_2\left(\wfq/(p^2)\right)$$
is of the form
$\left( \begin{array}{cc} \eps \psi & * \\0 & \psi^{-1}\end{array}\right)$ where the $\psi$ is unramified.
Thus the image of $\gamma$ contains a lifting
of $\left(\begin{array}{cc} -1 & 0\\0 & 1\end{array}\right) \in GL_2(\FF_{p^f})$
to an element of order $2p$ in $GL_2\left(\wfq/(p^2)\right)$ so the image
is full.
The cohomology class $f$ also gives rise to an  extension with
the same short exact sequence as above corresponding to the deformation to the dual numbers.
We show the $\ad$ parts of these extensions are different by considering the images of $G_p$.
The class $f \in H^1(G_{Y_N},\ad)$ belongs to the trace zero cohomology.
When comparing the field extension it cuts out over $\QQ(\rhob)$ to that of $\gamma$, it suffices to
twist $\gamma$ by a character $\phi$ so   $det(\gamma \otimes \phi)=\tilde{\eps}$, the Teichm\"{u}ller lift
of $\eps$. One easily sees that the image of inertia at $p$ in 
$\left((\gamma \otimes \phi) (G_p)\right)^{ab}$ is at least of order $p$ while
inertia at $p$ has image of order  $p-1$ in the corresponding abelianisation of the lift to the dual
numbers arising from $f$. Thus the $\ad$ extensions are distinct and
the
image of $\rho_{R^{ord,Y_N-new}_2} \mod (p,U)^2$ is full.  Lemma~\ref{Boston} implies
$\rho_{R^{ord,Y_N-new}_2} \mod (p^N,U^{Ne})$ has full image.

Now take
$\{ \phi_1,\dots,\phi_s\}$ and  $\{ h_1,\dots ,h_s,f \}$ as bases for
$H^1_{{\mclt}^{\perp}}(G_{X_N},Ad\rhob^*)$ and $H^1_{\mclt}(G_{X_N},Ad\rhob)$.
As before $f \in H^1_{\mcl}(G_{X_N},\ad)$ is the cohomology class arising from $\rho_n$ mod $(\pi^2)$.
Let $Q_i$ be the Chebotarev
set of primes $q_i$ satisfying
\begin{itemize}
\item $\phi_i|_{G_{q_i}} \neq 0$ (equivalently ${\phi}_{i,\ad^*} \neq 0$) , 
\item $h_i|_{G_{q_i}} \notin \mclt_{q_i}$ (equivalently ${h}_{i,\ad} \mid_{G_{q_i}} \neq 0$) and ${h}_{i,sc} \mid_{G_{q_i}} =0$, 
\item for $j \neq i$, $\phi_j,h_j|_{G_{q_i}} = 0$ and
\item $q_i$ is $\rho_N$-nice, that is $\rho_N(Fr_{q_i}) =\left(\begin{array}{cc} q_i & 0\\0 & 1\end{array}\right)$ where this element
has order prime to $p$. 
\end{itemize}
Setting $$\Gamma=\{ A \in \mbox{Image}(\rho_N) \mid A \equiv I \mod (p,U)\}$$
and using the fullness of $\rho_N$ established above,
one can easily adapt the proof of Lemma~\ref{shekharslemma} to show $\dim H^1(\mbox{Image}(\rho_N),\ad) =1$.
One then modifies Lemma~\ref{chebs} to show the above bullet points
are independent Chebotarev conditions.

Let $q_i \in Q_i$
and set $Y_N=X_N \cup \{q_1,\dots q_s\}$,
Then by part 2) of Propostion~\ref{nicefull} 
$\dim H^1_{\mclt^{\perp}}(G_{Y_N},Ad\rhob^*)=0$ and 
$\dim H^1_{\mclt}(G_{Y_N},Ad\rhob)=1$ and this last group has basis $\{f\}$.
The ordinary ring
$R^{ord,Y_N-new}\simeq\wfq[[U]]$.
The fourth condition on the $Q_i$ implies that the deformation $\rho_N$ arises from the
weight $2$ quotient of $R^{ord,Y_N-new}$ so $R^{ord,Y_N-new}_{2^*} \twoheadrightarrow\wfq[[U]]/(p^N,U^{Ne})$.
\end{proof}

The second surjection of ~\eqref{eq:seq} 
implies
$w_{2^*,N}(U) \in (p^N,U^{Ne})$ and Lemma~\ref{disting} 
implies $w_{2,N}(U)$ is a distinguished polynomial. It thus has  degree at least $Ne$.

\subsection{Cutting down the size of weight 2 deformation rings via local obstructions}

Our next step is to add more nice primes of ramification   so that the new 
weight $2$ ordinary  ring is a quotient of $\wfq[[U]]$  
by a polynomial $v_{2,N}(U)$ of 
degree  {\it exactly} $e$. Furthermore, $v_{2,N}(U)$ will have  a root
$y_{N,1}$ such that $U \mapsto y_{N,1}$ gives rise to $\rho_n$.

Let $C$ be a positive number smaller than both $|\pi^n|$ and the minimum
of half the distances between any pairs of roots of $g_{\pi}$,
its  Krasner bound.

Recall that $U \mapsto \pi$ in
$$G_{Y_N} \stackrel{\rho_{R^{ord,Y_N-new}}}{\longrightarrow}
 GL_2(\wfq[[U]]) \to GL_2\left(\wfq[[U]]/(p^N,U^{Ne})\right) \to GL_2\left(\mco/(\pi^n)\right).$$

Denote by $\rho_{g_{\pi},N}$ the deformation 
$$G_{Y_N} \to
 GL_2\left(R^{ord,Y_N-new}=\wfq[[U]]\right) \to GL_2\left(R^{ord,Y_N-new}_{2^*}\right)$$
$$ \to GL_2\left(\wfq[[U]]/(p^N,U^{Ne})\right)
\to GL_2\left(\wfq[[U]]/(p^N,g_{\pi}(U), U^{Ne})\right).$$
Let $\rho_{p,k}$ be the reduction of the deformation  
$G_{Y_N} \stackrel{\rho_{R^{ord,Y_N-new}}}{\longrightarrow}
GL_2\left(\wfq[[U]]/(p^N,U^{Ne})\right)$
mod $(p,U^k)$.

Note 
\begin{itemize}
\item $\dim H^1_{\mclt}(G_{Y_N},Ad\rhob)=1$ and this space has basis $\{f\}$,
\item $\dim H^1_{\mcl}(G_{Y_N},\ad)=1$ and this space has basis $\{f\}$,
\item $\dim H^1_{\mclt^{\perp}}(G_{Y_N},Ad\rhob^*)=0$ and
\item $\dim H^1_{\mcl^{\perp}}(G_{Y_N},\ad^*)=1$ and this space has some basis,
say $\{\phi\}$.
\end{itemize}

Let $Q$ be the set of primes $q$ satisfying
\begin{itemize}
\item $H^1(G_{Y_N},\ad)|_{G_q}=0$, 
\item  $q$ is $\rho_{g_{{\pi},N}}$-nice.
\item $\phi \mid_{G_q}\neq 0$, 
\item $\rho_{p,e+1}(Fr_q) =\left(\begin{array}{cc} q(1+U^e) & 0\\0 & 1-U^e\end{array}\right)$.
\end{itemize}
\begin{prop}\label{cheb2}

The Chebotarev conditions defining $Q$ above are  independent for $p \geq 3$.
\end{prop} 
\begin{proof}

The first two conditions are complete splitting conditions in fields 
above $\QQ(\rhob)$
and can therefore be satisfied simultaneously.
They are both $Ad\rhob$ conditions and thus independent of the third condition, an $Ad\rhob^*$ condition.
It remains to show the independence of the fourth condition from the previous three.
Since it is a succession of $Ad\rhob$ conditions, we only have to check independence with the first two
conditions and, since  $e>1$, independence with the first condition follows from
Lemma~\ref{shekharslemma}.

Finally we  show the independence of the fourth and second conditions.
Let $K_{g_{{\pi},N}}, K_e$ and $K_{e+1}$ be the fields fixed by the kernels of 
$\rho_{g_{{\pi},N}}$,
$\rho_{p,e}$ and $\rho_{p,e+1}$ respectively.

\vskip1em
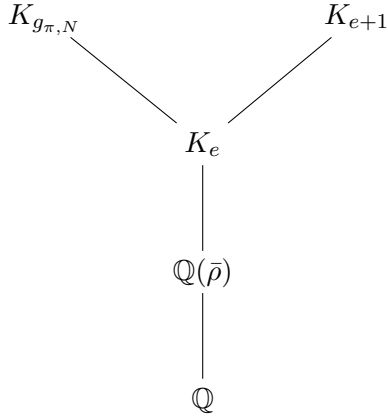
\hspace{5cm}
\begin{tikzpicture}[description/.style={fill=white,inner sep=2pt}]
\matrix (m) [matrix of math nodes, row sep=3em,
column sep=2.5em, text height=1.5ex, text depth=0.25ex]
{ K_{g_{\pi,N}}&  & K_{e+1}\\
& K_e & \\
 &\QQ(\rhob) & \\ 
& \QQ & \\};
\path (m-3-2) edge  (m-4-2);
\path (m-3-2) edge  (m-2-2);
\path(m-2-2) edge   (m-1-1);
\path(m-2-2) edge   (m-1-3);
\end{tikzpicture}
\newline\noindent
As $g_{\pi}(U)$ is distinguished of degree $e$,
$K_{g_{\pi,N}} \supset K_e$. 
We'll show $K_{{g_{{\pi},N}}} \not \supseteq
K_{e+1}$.   The same argument given in the proof of Lemma~\ref{YN}
implies $\rho_{R^{ord,Y_N-new}_{2^*}} \mod (p,U)^2$
has full image.  Lemma~\ref{Boston} then implies
the image of $\rho_{R^{ord,Y_N-new}_{2^*}} \mod (p^N,U^{Ne})$ contains
$\left(\begin{array}{cc} 1+g_{\pi}(U) &0\\0 & (1+g_{\pi}(U))^{-1}\end{array}\right)$.
When we reduce mod $(p^N,g_{\pi}(U),U^{Ne})$ to $\rho_{_{g_{{\pi},N}}}$ this element becomes trivial.
But when we reduce mod $(p,U^{e+1})$ to $\rho_{p,e+1}$, bearing in mind that
$g_{\pi}(U) \equiv U^e \mod p$, the image is 
$\left(\begin{array}{cc} 1+U^e & 0\\0 & 1-U^e\end{array}\right)$.
So $K_{{g_{{\pi},N}}} \not \supseteq
K_{e+1}$. Thus $K_{g_{{\pi},N}}$ and $K_{e+1}$ are linearly disjoint over $K_e$.
The  second condition is a complete splitting condition in $K_{g_{{\pi},N}}$ while
the fourth is a complete splitting condition in $K_e$, but {\bf not} in $K_{e+1}$.
\end{proof}

\subsection{Proof of Theorem~\ref{hard}}
Choose $q_1 \in Q$. Part 4) of Proposition~\ref{nice}, using
the first and third bullet points on $q_1$, implies
$$\dim H^1_{\mcl}(G_{ Y_N \cup \{q_1\}},\ad)=1 = \dim H^1_{\mclt^{\perp}}(G_{ Y_N \cup \{q_1\}},\ad^*)$$
and $ H^1_{\mcl}(G_{ Y_N \cup \{q_1\}},\ad)$ is spanned by $\{f\}$ and
$H^1_{\mcl^{\perp}}(G_{Y \cup \{q_1\} },\ad^*)$ 
is spanned by some $\{\tilde \phi\}$ ramified at $q_1$.
Thus $R^{ord,Y_N \cup \{q_1\}-new}_{2^*}$
is a quotient of
$\wfq[[U]]$ with one dimensional tangent space.

By part 1) of Proposition~\ref{nicefull} there are two possibilities:
\begin{itemize}
\item $\dim H^1_{\mclt}(G_{Y_N \cup \{q_1\} },Ad\rhob) =1$ or
\item $\dim H^1_{\mclt}(G_{Y_N \cup \{q_1\} },Ad\rhob) =2$.
\end{itemize}
In the first case, $R^{ord,Y_N\cup \{q_1\}-new} \simeq\wfq[[U]]$
and its weight $2^*$ quotient is formed by quotienting by the 
one determinant relation $v_{2^*,N}(U)$ which we can assume
is a distinguished polynomial  by Lemma~\ref{disting}. In the latter case it is possible that
$R^{ord,Y_N\cup \{q_1\}-new}_{2^*}$ is a quotient of $\wfq[[U]]$ by either multiple relations
or that it might not be finite and flat over $\wfq$. We will deal with this case by adding another prime
of ramification. 

While each $q_1 \in Q$ puts us in one of the two cases above, it is an open (and difficult!)
question if both cases can occur. Part of the length of the argument below is because of this.
That we do not know whether
we have to allow ramification at one or two nice primes to remove obstructions to deformation
problems here and in \cite{KLR} is the same phenomenon.

\subsubsection{Case 1}

In the  first case we have deformations associated to the ring homomorphisms
$$\wfq[[U]] \simeq  R^{ord,Y_N \cup \{q_1\}-new}  \twoheadrightarrow
R^{ord,Y_N \cup \{q_1\}-new}_{2^*} 
\simeq \wfq[[U]]/\left(v_{2^*,N}(U)\right) \twoheadrightarrow
\wfq[[U]]/\left(p^N,g_{\pi}(U),U^{Ne}\right).$$
The last surjection above implies 
$v_{2^*,N}(U) \in (p^N,g_{\pi}(U), U^{Ne})$ so its
degree  is
at least $e$. 
We claim it  is exactly $e$. 

If the degree is greater than $e$, then $R^{ord,Y_N \cup \{q_1\}-new}_{2^*} \twoheadrightarrow
\FF_{p^f}[[U]]/(U^{e+1})$. 
Call the corresponding  deformation $\alpha$ and let $\beta$ be the deformation induced by
the composite
$$R^{ord,Y_N-new}_2\twoheadrightarrow 
\wfq[[U]]/\left(p^N,U^{Ne}\right) \twoheadrightarrow
\FF_{p^f}[[U]]/(U^{Ne}) \twoheadrightarrow
\FF_{p^f}[[U]]/(U^{e+1}).$$ 
Note $\alpha\mid_{G_{q_1}} \in {\mathcal C}_{q_1}$ and 
$$\beta(Fr_{q_1}) = \left(\begin{array}{cc} q_1(1+U^e)&0\\0&1-U^e \end{array}\right) 
\implies \beta\mid_{G_{q_1}}
\notin {\mathcal C}_{q_1},$$
that is $\beta$ is not Steinberg at $q_1$.
As both $\alpha$ and $\beta$ are deformations of $\rho_{g_{\pi,N}}$ mod $p$ to 
$GL_2\left(\FF_{p^f}[[U]]/(U^{e+1})\right)$ they differ
by a $1$-cohomology class $k \in H^1(G_{Y \cup \{q_1\} },\ad)$,
that is $$\alpha= (I+U^ek)\beta.$$
If $k$ is unramified at $q_1$,
then $k$ inflates from $H^1(G_{Y_N},\ad)$. But $q_1$ was chosen so
$H^1(G_{Y_N},\ad)\mid_{G_{q_1}} =0$. Thus $k$ cannot
change the local at $q_1$ deformation where
$\beta(Fr_{q_1})= \left(\begin{array}{cc} q_1(1+U^e)&0\\0&1-U^e \end{array}\right)$
to one in ${\mathcal C}_{q_1}$. So $k$ is ramified at $q_1$.
But we chose $q_1$ such that $\phi\mid_{G_{q_1}} \neq 0$ where $\phi$ spanned
$H^1_{\mcl^{\perp}}(G_{Y_N},\ad^*)$.
Parts 1) and 2) of Proposition~\ref{nice} then imply  the
map 
$$H^1(G_{Y_N \cup \{q_1\}},\ad) \to\oplus_{v \in Y_N} \frac{H^1(G_v,\ad)}{\mathcal L_v}$$
has image one dimension larger than the map
$$H^1(G_{Y_N},\ad) \to\oplus_{v \in Y_N} \frac{H^1(G_v,\ad)}{\mathcal L_v}.$$
For all $v \in Y_N$ we have $\alpha\mid_{G_v}$ belongs to our deformable class ${\mathcal C}_v$
as does
$\beta\mid_{G_v}$. But for at least one $v$ we have $k_v \mid_{G_v} \notin \mathcal L_v$
so $\alpha\mid_{G_v}=(I+U^ek)\beta \mid_{G_v} \notin {\mathcal C}_v$, a contradiction.  
Thus $k$ can be neither ramified nor unramified at $q_1$. This contradiction implies
$v_{2^*,N}(U)$ has degree $e$ mod $p$.

Recall $v_{2^*,N}(U) \in (p^N,g_{\pi}(U),U^{Ne})$ so
\begin{equation}\label{eq:gpi}
v_{2^*,N}(U) =a(U)p^N +b(U)g_{\pi}(U)+c(U)U^{Ne}.
\end{equation}
Recall that $g_{\pi}(U)$ is the minimal polynomial over $W(\FF_{p^f})$ of $\pi$ and that 
its roots are distinct.
Since both $g_{\pi}(U)$ and $v_{2^*,N}(U)$ are degree
$e$,  $b(U)$ is a unit. 
Let $\{y_{N,1},y_{N,2},\dots,y_{N,e}\}$
be the roots of $v_{2^*,N}(U)$. As $v_{2^*,N}(U)$ is distinguished of degree $e$, $v_p(y_{N,i}) \geq 1/e.$
Observe
$$0=v_{2^*,N}(y_{N,i}) = p^Na(y_{N,i}) +b(y_{N,i})g_{\pi}(y_{N,i}) +c(y_{N,i})y^{Ne}_{N,i}.$$
The outside terms on the right  have valuation at least $N$ and $b(y_{N,i})$ is a unit,
so $v_p(g_{\pi}(y_{N,i})) \geq N$.
Thus  $y_{N,i}$ is very close to a root of $g_{\pi}(U)$. 
For $N$ large enough, this closeness
is closer than the common Krasner bound $C$ on the roots of $g_{\pi}(U)$.
We claim each $y_{N,i}$ is close to a different root of $g_{\pi}(U)$.  If this were false then
a root of $g_{\pi}$ would be missed, that is 
there would be a root $x_0$ of $g_{\pi}(U)$ with $|x_0-y_{N,i}| > C$ for all $i$.
As $v_{2^*,N}(U) = \prod (U-y_{N,i})$, we would have $|v_{2^*,N}(x_0)| > C^e$.
Evaluating ~\eqref{eq:gpi} at $x_0$ gives $|v_{2^*,N}(x_0)| < p^{-N}$, a contradiction for large $N$ so
the claim is true.
After relabelling,
we may assume $y_{N,1}$ is close to $\pi$.

The composite deformations corresponding to
$$\wfq[[U]]\simeq
 R^{ord,Y_N \cup \{q_1\}-new}\twoheadrightarrow R^{ord,Y_N \cup \{q_1\}-new}_{2^*}
\twoheadrightarrow \wfq[[U]]/(p^N,g_{\pi}(U),U^{Ne})$$
and
$$\wfq[[U]] \simeq
R^{ord,Y_N -new} \twoheadrightarrow R^{ord,Y_N -new}_{2^*}
\twoheadrightarrow \wfq[[U]]/(p^N,g_{\pi}(U),U^{Ne})$$
are the same as the latter is nice at $q_1$.
We know $U \mapsto \pi$ in the latter to give  $\rho_n$
so sending $U$ to $\pi$ in the former gives $\rho_n$ as well.
As $y_{N,1}$ is  close enough to $\pi$, Krasner's lemma implies
$\wfq[\frac1p](\pi) \subset \wfq[\frac1p](y_{N,1})$. 
As  $\left[\wfq[\frac1p](y_{N,1}):\wfq[\frac1p]\right] \leq \mbox{deg} (v_{2^*,N}(U)) =e$, 
the fields $\wfq[\frac1p](y_{N,1})$ and
$\wfq[\frac1p](\pi)$
are equal. Recall that $C$ is smaller than both $|\pi^n|$ and the Krasner bound on the roots of $g_{\pi}(U)$.
We  chose $N$ large enough so that
$|y_{N,1}-\pi| < C < |\pi^n| $, so sending $U$ to $y_{N,1}$ in the former sequence gives $\rho_n$ as well.
As $R^{ord,Y_N\cup\{q_1\}-new}_{2^*} \simeq \wfq[[U]]/(v_{2^*,N}(U))$, we see $\rho_n$ lifts
to an $\mco$-valued weight $2$ Galois representation.
This proves Theorem~\ref{hard} in the case where we assumed
that 
$\dim H^1_{\mcl}(G_{Y_N \cup \{q_1\} },Ad\rhob)=1$ which implied
$R^{ord,Y_N \cup \{q_1\}-new} \simeq \wfq[[U]]$. In this case we set $T_2 = T \cup \{q_1\}$.

\subsubsection{Case 2}

We now deal with the second more involved case.
Had we allowed ourselves standard `$R=T$' theorems, 
then we could assume $R^{Y_N \cup \{q_1\}}_2$ is a finite
flat complete intersection and proved the second half of Theorem~\ref{hard},
$$R^{ord, Y_N \cup \{q_1\}-new} \twoheadrightarrow \mco \twoheadrightarrow \mco/(\pi^n)$$
exactly as in case 1). The arbitrary weight ordinary ring might still be a quotient of a power series ring in two variables.

We have:
\begin{itemize}
\item $\dim H^1_{\mclt}(G_{Y_N \cup \{q_1\} },Ad\rhob)=2$, a basis for this 
         space is $\{f,h\}$ where $h \in H^1(G_{Y \cup \{q_1\} },Ad\rhob) \backslash H^1(G_{Y_N  },\ad)$ and 
         is ramified at $q_1$,
\item $\dim H^1_{\mclt^{\perp}}(G_{Y_N \cup \{q_1\} },Ad\rhob^*)=1$ and a basis for this space is $\{\psi\}$ and
         $\psi$ is ramified at $q_1$,
\item $\dim H^1_{\mathcal L}(G_{Y_N \cup \{q_1\} },\ad)=1$ and  a basis for this 
         space is $\{f\}$,
\item $\dim H^1_{{\mathcal L}^{\perp}}(G_{Y_N \cup \{q_1\} },\ad^*)=1$ and a basis for this 
         space is $\{\tilde \phi\}$ where $\{\phi,\tilde \phi\}$ is independent and $\tilde\phi$ is ramified at $q_1$.
         Recall $\{\phi\}$ formed a basis of $H^1_{\mclt^{\perp}}(G_{Y_N},Ad\rhob^*)$ and $\phi \mid_{G_{q_1}} \neq 0$ implies
         $\phi \mid_{G_{q_1}} \notin {\mcl}^{\perp}_{q_1}.$
\end{itemize}

We chose $q_1$ to satisfy Proposition~\ref{cheb2}.
Now choose a second prime $q_2$ such that

\begin{itemize}
\item 
          $h_{\ad} \neq 0$ and $h_{sc} =0$ so $h\mid_{G_{q_2}} \notin \mclt_{q_2}$,
\item $\psi,\tilde\phi \mid_{G_{q_2}} \neq 0$,
\item $H^1(G_{Y_N},\ad^*) \mid_{G_{q_2}}=0$,
\item $ H^1(G_{Y_N},\ad)\mid_{G_{q_2}} =0$,
\item  $q_2$ is nice for  $\rho_R:G_{Y_N}\to GL_2\left(\wfq[[U]]\right) \to GL_2\left(\wfq[[U]]/(p^N,U^{Ne})\right)$.
\end{itemize}

We have already remarked that the set $\{\phi,\tilde \phi\}$ is independent.
So is $\{\psi,\phi\}$ as  $\psi$ is ramified at $q_1$ but $\phi $ is not.
If $\psi$ is not in the span of $\{\phi,\tilde\phi\}$, then the second
and third bullet points are independent Chebotarev conditions.
If $\psi$ is  in the span of $\{\phi,\tilde\phi\}$, then choosing $q_2$
so that $\phi \mid_{G_{q_2}} = 0$ and $\tilde \phi \mid_{G_{q_2}}  \neq 0$ implies
$\psi \mid_{G_{q_2}}  \neq 0$. The above bullet points give independent Chebotarev conditions.

Set $T_2 =Y_N \cup \{q_1,q_2\}$.
Then, by our choice of $q_2$ and part 2) of Proposition~\ref{nicefull}
\begin{itemize}
\item $\dim H^1_{\mclt}(G_{T_2},Ad\rhob)=1$ and a basis for this 
         space is $\{f\}$,
\item $\dim H^1_{\mclt^{\perp}}(G_{T_2},Ad\rhob^*)=0$,
\item $\dim H^1_{\mathcal L}(G_{T_2},\ad)=1$ and  a basis for this 
         space is $\{f\}$,
\item $\dim H^1_{\mathcal L}(G_{T_2},\ad^*)=1$ and a basis for this 
         space is $\left\{\tilde{\tilde \phi}\right\}$.
\end{itemize}

We have that $R^{ord, T_2-new} \simeq \wfq[[U]]$ and again  by \cite{DT}
this is a Hida family. Its weight $2$ quotient is the middle term of
\begin{equation}\label{eq:ku}
\wfq[[U]]=R^{ord, T_2-new} \twoheadrightarrow R^{ord, T_2 -new}_{2^*}
=\wfq[[U]]/(m_{2^*,N}(U)) \twoheadrightarrow \wfq[[U]]/(p^N,g_{\pi}(U),U^{Ne})
\end{equation}
where $m_{2^*,N}(U) \in (p^N,g_{\pi}(U), U^{Ne})$ is a distinguished polynomial of degree at least $e$. Thus
$$R^{ord, T_2-new}_{2^*}
=\wfq[[U]]/(m_N(U)) \twoheadrightarrow \FF_{p^f}[[U]]/(U^e)$$ 
as both $q_1$ and $q_2$ were chosen
to be nice for this representation to $GL_2\left( \FF_{p^f}[[U]]/(U^e) \right)$, 
which is unramified outside $Y_N$. 
We prove by contradiction that deg$(m_N(U))=e$ by contradiction.
  Suppose degree$(m_{2^*,N}(U)) >e$.
Then we have  deformations
$$\alpha:G_{T_2} \to GL_2(\FF_{p^f}[[U]]/(U^{e+1})),\,\,\,
\beta:G_{Y_N} \to GL_2(\FF_{p^f}[[U]]/(U^{e+1}))$$ 
where $\alpha = \beta$ mod $U^e$ so they differ by a $1$-cohomology class.
We write this class as $xh_0+yh_{q_1}+zh_{q_2}$
where $h_0 \in H^1(G_{Y_N},\ad)$, $h_{q_i} \in H^1(G_{Y_N \cup \{q_i\}},\ad)$ for
$i=1,2$.
We assume $h_{q_i}$ is ramified
at $q_i$ for $i=1,2$.
Note $\alpha,\beta \mid_{G_v} \in {\mathcal C}_v$
for all $v \in Y_N$.

Since $\phi\mid_{G_{q_1}} \neq 0$, $\tilde{\phi}\mid_{G_{q_2}} \neq 0$ parts 1) and 2) 
of Proposition~\ref{nice} imply
the maps
\begin{equation}\label{eq:qyn}
H^1(G_{Y_N \cup \{q_1\}},\ad) \to \oplus_{v \in Y_N} \frac{H^1(G_v,\ad)}{\mathcal L_v}
\end{equation}
and
$$H^1(G_{Y_N \cup \{q_1,q_2\} },\ad) \to \oplus_{v \in Y_N\cup \{q_1\}} \frac{H^1(G_v,\ad)}{\mathcal L_v}$$
have images one dimension larger than
\begin{equation}\label{eq:zq2}
H^1(G_{Y_N },\ad) \to \oplus_{v \in Y_N} \frac{H^1(G_v,\ad)}{\mathcal L_v}
\end{equation}
and
$$H^1(G_{Y_N\cup \{q_1\}},\ad) \to \oplus_{v \in Y_N\cup \{q_1\}} \frac{H^1(G_v,\ad)}{\mathcal L_v}$$
respectively. So successively allowing ramification at  $q_1$ 
and then $q_2$ makes the Selmer maps more surjective by one dimension at each stage.
But as $H^1(G_{Y_N},\ad^*) \mid_{G_{q_2}} =0$, part 5) of Proposition~\ref{nice} again gives that the map
\begin{equation}\label{eq:ynq2}
H^1(G_{Y_N \cup \{q_2\}},\ad) \to \oplus_{v \in Y_N} \frac{H^1(G_v,\ad)}{\mathcal L_v}
\end{equation}
has the same image as
\begin{equation}\label{eq:yn}
H^1(G_{Y_N},\ad) \to \oplus_{v \in Y_N} \frac{H^1(G_v,\ad)}{\mathcal L_v}.
\end{equation}
If $h_{q_2} \mid _{G_{q_2}} 
\in {\mathcal L}_{q_2}$, then
allowing ramification at $q_1$ would have to add $2$ dimensions of surjectivity
to the map
$$H^1(G_{Y_N\cup\{q_1,q_2\}},\ad) \to \oplus_{v \in Y_N\cup \{q_1,q_2\}} \frac{H^1(G_v,\ad)}{\mathcal L_v}$$
compared to the map
$$H^1(G_{Y_N \cup q_2},\ad) \to \oplus_{v \in Y_N \cup q_2} \frac{H^1(G_v,\ad)}{\mathcal L_v}.$$
This is impossible by part 1) of Proposition~\ref{nice}
so $h_{q_2} \mid_{G_{q_2}} \notin {\mathcal L}_{q_2}$. (This idea was used already
in \cite{LR}).

So when we move from $\beta$ to $\alpha$ in the equation
$$\alpha = \left(I+( xh+yh_{q_1}+zh_{q_2}  )U^e\right)\beta,$$
what happens at $v \in Y_N$? 
The image of $h_{q_2}$ under
~\eqref{eq:ynq2} was already in the image of ~\eqref{eq:yn}. From ~\eqref{eq:zq2}, the image of $h_{q_1}$
lies outside this common image of ~\eqref{eq:ynq2} and~\eqref{eq:yn}.
Thus if $y \neq 0$ then for some $v_0 \in Y_N$ we have
$(xh+yh_{q_1}+zh_{q_2}) \mid_{G_{v_0}} \notin \mcl_{v_0}$
so either
$\alpha \mid_{G_{v_0}}\notin {\mathcal C}_{v_0}$
or $\beta \mid_{G_{v_0}}\notin {\mathcal C}_{v_0}$,
 a contradiction.
 Thus $y=0$. 

Now we consider properties at $G_{q_2}$. 
By the choice of $q_2$, both $\alpha,\beta \mid_{G_{q_2}} \in {\mathcal C}_{q_2}$
so  $(xh_0+zh_{q_2}) \mid_{G_{q_2}} \in {\mathcal L}_{q_2}$. But the choice of $q_2$
forces $h_0\mid_{G_{q_2}}=0$ and we showed that $h_{q_2} \notin {\mathcal L}_{q_2}$ so $z=0$.

Finally, we look locally at $q_1$. 
Recall $\beta(Fr_{q_1}) =\left(\begin{array}{cc} q_1(1+U^e) & 0 \\
0 & (1+U^e)^{-1}\end{array}
\right)$ so $\beta \notin {\mathcal C}_{q_1}$.
But $h_0\mid_{G_{q_1}}=0$ so $h_0$ cannot move $\beta$ to $\mathcal C_{q_1}$. We have a contradiction so $m_N(U)$ has degree
$e$. From ~\eqref{eq:ku} $m_N(U) \in (p^N,g_{\pi}(U),U^{Ne})$. Now simply proceed 
as in the Case $1$ with
$N$  suitably large.
\hfill$\square$

\subsection{Further corollaries}

\begin{cor} \label{weight} Let $\rhob:\gal \to GL_2(\FF_{p^f})$ be odd, full, ordinary, weight $2$ and have determinant $\eps$. 
 Let $\mco$ be any  totally ramified extension of $\wfq$.
There exists a set of primes $Y \supset S_0$ such that 
$$\wfq[[U]] \simeq R^{ord,Y} \twoheadrightarrow R^{ord,Y-new}_{2^*} \simeq \wfq[[U]]/(h(U))
\simeq \mco.$$ The degree of the map to weight space along the Hida family
$R^{ord,Y-new}$ is $[\mco:\wfq]$
when $\rhob$ is as in cases $1$), $2$) and $4$) of Proposition~\ref{vequalsp}.
In the other cases, the degree is strictly greater than $[\mco:\wfq]$.
There exists a weight $2$ form associated to $\rhob$ whose 
completed field of Fourier coefficients has ring of integers $\mco$.
\end{cor}
\begin{proof}
Use \cite{R3} to get a nontrivial weight $2$ deformation of $\rhob$ to 
$$\wfq[[U]]/(p,U^2) \simeq \mco/(\pi^2),$$  that is the corresponding cohomology class in this deformation
to the dual numbers is nonzero.
Now apply Theorem~\ref{hard}.
\end{proof}
\noindent

\begin{cor} \label{choosering}
Let  $\rhob:\gal \to GL_2(\FF_{p^f})$ be odd, full, ordinary, weight $2$ and have determinant $\eps$. 
Let $g(U) \in \wfq[U]$
be a distinguished  polynomial of degree $e$ with distinct roots and let $\epsilon >0$ be given. Then there exists a set 
$Y \supset S_0$  such that
$$\wfq[[U]] \simeq R^{ord,Y-new} \twoheadrightarrow R^{ord,Y -new}_{2^*} 
\simeq \wfq[[U]]/(w_{2^*}(U))$$ where 
$w_{2^*}(U)$ has degree $e$ and each
root of $w_{2^*}(U)$ is  within $\epsilon$ of a root of $g(U)$.
Furthermore if $g(U) = \prod g_i(U)$ where $g_i(U)$ is irreducible over $\wfq$ of degree $e_i$, then
$w_{2^*}(U) = \prod w_{2^*,i}(U)$ where $w_{2^*,i}(U)$ is irreducible over $\wfq$ of degree $e_i$ and
its roots are within $\epsilon$ of the roots of $g_i(U)$.
\end{cor}
\begin{proof} 
First choose $\epsilon$ to be less than half the distance between any pair of roots
of $g(U)$.
Use \cite{R3} to get a nontrivial weight $2$ deformation of $\rhob$ to 
$\wfq[[U]]/(p,U^2)$. Now proceed as in Proposition~\ref{vli} to get a weight $2$ deformation ring 
surjecting onto $\wfq[[U]]/(p^N,U^{Ne})$ with one dimensional tangent
space. Then add more primes so that
the ordinary ring is $\wfq[[U]]$ and its weight $2$ quotient is
$\wfq[[U]]/(w_{2^*}(U))$ where $w_{2^*}(U) \in (p^N,g(U),U^{Ne})$ is degree $e$. 
The argument in case 1) of Theorem~\ref{hard} implies that for
 $N$ large enough,
each root of $w_{2^*}(U)$ is within $\epsilon$ of a distinct root of $g(U)$. By the choice of $\epsilon$ the roots of $w_{2^*}(U)$ 
are distinct. As $g_i(U)$ is a degree $e_i$ irreducible factor of $g(U)$, let 
$\{r_{i1},\dots,r_{ie_i}\}$ be its roots and we know $|r_{ij} -s_{ij}| < \epsilon$ where $s_{ij}$ is a root of $w_{2^*}(U)$. 
We write $r_{ij} = s_{ij}+x_{ij}$.
Let $\sigma$ be an automorphism taking $r_{ij}$ to $r_{ik}$. Then 
$$r_{ik}=\sigma(r_{ij}) =\sigma(s_{ij}+x_{ij}) = \sigma(s_{ij}) +\sigma(x_{ij}).$$
As $\sigma$ preserves sizes, $$|r_{ik}-\sigma(s_{ij}) | =|\sigma(x_{ij})| =|x_{ij}| < \epsilon ,$$
so $\sigma(s_{ij})$ is the root of $w_{2^*}(U) \in \wfq[U]$  close to $r_{ik}$. Thus
$\sigma(s_{ij})=s_{ik}$ and the roots of $w_{2^*}(U)$ break up into Galois orbits corresponding
to the Galois orbits of the roots of $g(U)$ that are close to them. This proves the factorisation statement.
\end{proof}
\noindent
{\bf Remarks:}

1. We have the following lemma whose proof we owe to N.~Fakhruddin. This implies that Corollary \ref{choosering} gives examples of deformation rings which are non-integrally closed orders in valuation rings. Ralph Greenberg had asked one of us if there were such examples.

\begin{lemma}\label{naf}
 Let $f$ be a monic polynomial in $\Z_p[X]$ with distinct roots. Then  for all monic polynomials $g \in \Z_p[X]$  of degree $n$ which are close enough to $f$,   we have an isomorphism of $\Z_p$-algebras $$\Z_p[X]/(f) \simeq \Z_p[X]/(g).$$
\end{lemma}

\begin{proof}
 Let $n$ be the degree of $f$,  and consider the map $\alpha:\Z_p^n \rightarrow \Z_p^n$
 defined as follows. Given $\gamma \in \Z_p^n$ we regard it as an element of $\Z_p[X]/(f)$, and send it to the tuple  $(a_1,\cdots,a_n)$ in $\Z_p^n$,  where  the characteristic polynomial of the endomorphism of $\Z_p[X]/(f)$ given by multiplication by $\gamma$,  is $x^n+a_1x^{n-1}+a_2x^{n-2}+\cdots+a_n$. The image of $x$  under $\alpha$ is given by the coefficients of $f$. Using that $f$ has distinct roots, we see that $\alpha$ is an open mapping in a neighborhood of $x$, and deduce that all elements   in a sufficiently small  neighborhood  $V$ of $\alpha(x)$ are in   $\alpha (U)$ where $U$ is the open neighborhood of $x \in \Z_p^n$ consisting of elements that are congruent to $x$ mod $p$. We may assume that elements of $V$ correspond to monic polynomials  of degree $n$ in $\Z_p[X]$ with distinct roots.
From this we deduce  that for $g$ close enough to $f$, we  get a monomorphism    $\Z_p[X]/(g) \ra \Z_p[X]/(f)$ of $\Z_p$-algebras with finite cokernel, and  further deduce this is an isomorphism by reducing mod $p$.
\end{proof}

One could ask, like in Krasner's lemma, for quantitative refinements of this lemma.

2.  Let $a \in W(\FF_{p^f})$
be a nonsquare in $W(\FF_{p^f})$. Then if one chooses $g(U)=U^2-ap^2$ in Corollary~\ref{choosering},
the deformation ring will be an {\it order} in the unramified degree $2$ extension of $W(\FF_{p^f})$.
Thus one can also obtain nontrivial  unramified extensions as completed fields of Fourier coefficients
of the modular form corresponding to our Galois representation.

\vspace{3mm}


\section*{APPENDIX: Modularity of geometric lifts $\rho$ via $p$-adic approximation}
We apply  Theorem \ref{easy}
to proving modularity of  certain representations $\rho:\gal \rightarrow GL_2(\mco)$ where $\mco$ is the valuation ring of a finite extension $K$ of $\QQ_p$ and $p>2$.  

\begin{theorem}\label{approximative} Let $\rho:\gal \to GL_2(\mco)$ be  odd, ordinary, full, balanced,  ramified at only finitely many primes, weight $2$ with determinant $\eps$  and have modular reduction. Further assume that when $\rhobar$ is split at $p$, $\rho$ is flat at $p$. Then
$\rho$ is modular.
\end{theorem}

 The method of proof   extends   the method of \cite{K} which dealt with the case when  $K$ is unramified over $\QQ_p$. Of course these results are contained in those of the various `$R=T$' theorems
pioneered by Wiles and Taylor-Wiles. Our point here is to provide a 
different argument  using $p$-adic approximation. In  this appendix the proofs are merely sketched: we are rederiving known results  using Theorem \ref{easy} and the strategy of \cite{K}.

Using  Theorem \ref{easy}, we first prove  that  for each $n$, $\rho_n=\rho$ mod $(\pi^n)$, 
is modular of a level which depends on $\rho_n$.
We  then lower the level of $\rho_n$ which paves the way to proving modularity
of  $\rho:\gal \rightarrow GL_2(\mco)$ by successive approximation. 

We have the following corollary of Theorem~\ref{easy}.  We keep the notation of the previous sections: for instance $S_0$ is the set of primes at which $\rhobar$ is ramified, and $S$ is the set of primes at which $\rho$ is ramified.


\begin{cor}
 The representation  $\rho_n$ is isomorphic  as $\Oc[G_\Q]$-representation to a submodule of the ordinary part of the  $p$-divisible group associated to $J_T$ tensored over $\Z_p$ with $\Oc$. Here $J_T$ is the  Jacobian of the  projective  modular curve $\Gamma_1(Np^r) \cap \Gamma_0(Q_n) $; $N$ is some fixed integer independent of $n$;  $r$ is an integer which a priori depends on $n$; and $Q_n$ is the product of primes in the finite set $ T\backslash S$ which depends on $n$. 
\end{cor}

We say that $\rho_n$  as in the corollary arises from $J_T$.
We also say that $\rho_n$ arises from the $p$-divisible group associated to
the ordinary factor  $J_T^{ord}$ of $J_T$.

\begin{proof}
This follows from  Theorem~\ref{easy},  level raising results of \cite{DT} and  Hida's theory. 
These results yield that $R^{ord,T-new} $ surjects onto a  $T$-new ordinary  Hecke algebra 
$\T^{ord,T-new}$ which is finite and torsion free over $\Lambda=\ZZ_p[[T]]$. But using that $R^{ord,T-new} \simeq\wfq[[U]]$ we deduce that we have  an isomorphism of $R^{ord,T-new} $  with $\TT^{ord,T-new}$.
This  implies the corollary  by standard arguments.
\end{proof}

Proposition~\ref{qold}  follows from arguments in  
  \cite{K} with the twist that as we allow 
primes $q \in T \backslash S_0$ that are $-1$ mod $p$, we have to  keep track of the  Atkin-Lehner operators $W_q$ for $q \in  Q_n=T \backslash S_0$. 
We have the relation $W_q^2=\langle q \rangle$.

\begin{prop}\label{qold}
The representation $\rho_n$   arises from the $Q_n$-old subvariety of $J_T^{ord}$, 
and furthermore all the Hecke operators $T_r$, for $r$ a prime not in $T$, act on $\rho_n$ by $Trace(\rho_n(\rm{Frob}_r))$.
\end{prop}

\begin{proof}
This is an application of Mazur's principle (see \S 8 of \cite{Ken}) and uses that $q \not \equiv 1$ mod $p$. 
The principle relies on the fact that 
the ${\rm Frob}_q$-action
on  unramified finite $G_{\Q_q}$-submodules  of the torsion points of $J_T$  whose reduction mod $q$ is in the   `toric part' of the reduction mod $q$ of  $J_T$ is constrained. 
Namely,
on the `toric part' the Frobenius ${\rm Frob}_q$ acts by  $-qW_q$ where $W_q$  is the Atkin-Lehner involution. We flesh this out this below.

Consider a prime  $q \in Q_n$ where $Q_n=T \backslash S$. Then decompose $\rho_n|_{D_q}$  (which is unramified
by hypothesis) into $\mco/(\pi^n) \oplus \mco/(\pi^n)$, with basis $\{e_n,f_n\}$  with ${\rm Frob}_q (e_n)= -W_q.e_n$  and ${\rm Frob}_q (f_n)=-qW_q.f_n$ for some character $\epsilon$ as above.  The action of $W_q$ is by a scalar $\alpha_q$, and so we have ${\rm Frob}_q (e_n)= - \alpha_qe_n$  and ${\rm Frob}_q (f_n)=-q\alpha_qf_n$ .

Using irreducibility of  $\rho$, Burnside's lemma gives that $\rho(\FF_{p^f}[\gal]) = M_2(\FF_{p^f})$ 
  and
hence by Nakayama's lemma $\rho_n(\mco[\gal]) = M_2(\mco/(\pi^n))$. 
Thus using the surjection from the Hecke algebra acting on $J_T^{ord}$ to $\mco/(\pi^n)$ 
we deduce that $\rho_n$ arises from  an eponymous submodule of $J:=J_T^{ord}$.

The fact that the $q$-old subvariety is stable under the Galois and Hecke
action will allow us to  deduce that $\rho_n$  arises from  the $q$-old subvariety of $J$  if  we can show that $e_n$  is contained in the $q$-old subvariety of $J$.

 Let $\mathcal J$
be the N\'eron model at $q$ of $J$. Note that as $\rho_n$ is unramified at $q$ it
maps injectively to ${\mathcal J}_{/\FF_{p}}(\bar \FF_{p}) $ under the reduction map. Now if the claim
were false, as the group of connected components  of $\mathcal J$ is Eisenstein,  
we would deduce that the reduction of $e_n$ in ${\mathcal J}^0 (\bar \FF_{p})$ maps non-trivially (and hence its image has order divisible by $p$) to the $\bar \FF_{p}$-points of the torus which is the quotient of ${\mathcal J}^0$  by the image of the $q$-old subvariety
(in characteristic $q$). But as we recalled above, it is well-known (see \S $8$ of \cite{Ken}) 
that  ${\rm  Frob}_q$ acts on the $\bar \FF_p$-valued points of this toric quotient (isogenous to the torus $T$  of $\mathcal J^0$ , the latter being a semiabelian variety that is an extension of  an abelian variety  by $T$) by  $-qW_q$ which gives a contradiction.  Now taking another prime  $q' \in Q_n$ and working within the $q$-old subvariety of  $J$, by the same argument we see that $\rho_n$  occurs in the $\{q,q'\}$-old subvariety of $J$, and eventually that $\rho_n$  occurs in the $Q_n$-old subvariety of $J$.  The last part of the proposition is 
then clear.
\end{proof}

We finish the proof of the main theorem of the appendix Theorem \ref{approximative}.

\begin{proof}
For an integer $N$ prime to $p$, denote by $J_1(Np^\infty)^{ord}$ the direct limit of the ordinary parts of $J_1(Np^r)$ as $r$ varies. From Proposition~\ref{qold} it is easy to deduce that  $\rho_n$, the mod $(\pi^n)$ reduction of $\rho$,  
arises from $J_1(Np^\infty)^{ord}$  for some fixed integer $N$ that is independent of $n$. Let  $\T$ be the Hida  Hecke 
algebra acting on $J_1(Np^\infty)^{ord}$, generated by the Hecke operators $T_r$ with $ r$ prime and prime to $Np$. 
We claim that the $\rho_n$  give compatible morphisms from $\T$  to the $\mco/(\pi^n)$. To get 
these morphisms, let $V_n$ denote a realisation of the representation $\rho_n$ in  $J_1(Np^\infty)^{ord}$ which 
exists by Proposition~\ref{qold}. Then $V_n$ is $\gal$-stable, and hence $\T$-stable (because of the 
Eichler-Shimura congruence relation mod $r$, that gives an equality of 
correspondences  $T_r = {\rm Frob}_r +r \langle r \rangle{ {\rm Frob}_r}^{-1}$ where ${\rm Frob}_r$ is 
the Frobenius morphism at $r$). So $V_n$ is a $\T$-module, and because of the absolute 
irreducibility (only the scalars commute with the $\gal$-action) $\T$ acts via a morphism 
$\alpha_n : \T \rightarrow \mco/(\pi^n)$ as desired, and the $\alpha_n$'s  are compatible 
again because of the congruence relation. This gives a morphism 
$\alpha: \T \rightarrow  \mco$ such that the  representation associated to $\alpha$ is isomorphic
to $\rho$ which finishes proof of  the theorem.  Then using that the determinant of $\rho$ is $\eps$ and Hida's control theorem, we deduce that $\rho$ arises from a weight 2 newform.

\end{proof}

\subsection*{Improvements to the method}

The weight 2 assumption on $\rhobar$ and  the lifts we consider (and the assumption on the determinant)  is  for convenience, and our  methods  apply to  $\rho$ of weight $k \geq 2$ (the Hodge-Tate weights are $(k-1,0)$),  provided that $\rhobar$ is distinguished at $p$.

The {\it fullness} assumption  on $\rho$ and $\rho_n$  used in the proof of Theorem \ref{easy} arises from the fact that
in its absence Lemma \ref{shekharslemma} is not true. On the other hand one can prove a more qualitative   but less restrictive version of this lemma. 

 \begin{lemma} \label{shekharslemma2} Let $p \geq 3$. Recall  $\Oc$ is the ring of integers of a finite extension of $\Q_p$, with
uniformiser $\pi$ and residue field $\FF_{p^f}$. Let $G \subset GL_2(\mco)$ be a closed subgroup. Assume that the image $G_1$  of $G \to GL_2(\FF_{p^f})$, contains $SL_2(\FF_{p^f})$ and satisfies hypothesis of Lemma \ref{p35}.   Then 
$\dim H^1(G,\ad) $ is a finite abelian group. 
\end{lemma}

The proof uses  (i) $H^1(G_1,\ad)=0 $ (cf. Lemma \ref{p35}), and (ii)  the kernel of the homomorphism  $G \rightarrow G_1$ is a finitely generated pro-$p$ group.

G.~B\"ockle  has observed that  using such a lemma,  one can remove the assumption on fullness of image of $\rho$, made in the arguments in the appendix,  by using base change to totally real solvable extensions $F/\Q$ and considering nearly ordinary deformations of $\rhobar|_{G_F}$  and Hida's nearly ordinary Hecke algebras.   Choose a totally real solvable extension $F/\Q$ disjoint from the field cut out by $\rho$, and  whose degree $d=[F:\Q]$  is $>\dim H^1(G,\ad) $.  Then by  choice of $F$ and the technique of killing dual Selmer groups, and obtaining smooth quotients of deformation rings of  the expected dimension of this paper, one obtains  nearly ordinary deformation rings that are power series rings in $d+\delta+1$ variables, where $d$ is the degree of $F$ over $\Q$ and $\delta$ the Leopoldt defect for $F$ and $p$,  and such that $\rho$ mod $\pi^n$ arises from the corresponding universal deformation. One then  would exploit the fact that Hida's nearly ordinary Hecke algebra is  finite flat over $\Z_p[[X_0,\cdots,X_{d+\delta}]]$.  By  more elaborate level lowering methods (as in \cite{Thorne})  one would then  by a similar strategy as above prove that $\rho|_{G_F}$ is automorphic which suffices as $F/\Q$ is solvable.

To make the present method of modularity lifting  more robust,  one would ultimately hope to also remove  the conditions of {\it ordinarity} and being {\it balanced} on geometric $\rho$, and show assuming $\rhobar$ is modular and irreducible, that $\rho$  mod $(\pi^n)$ is modular  of some level for each $n$, and hence  by level lowering techniques deduce that $\rho$ itself arises for a new form. For this  again base change to solvable totally real  extension of $\Q$ and the   
{\it trivial primes} of \cite{HR} might be useful to remove the {\it balanced condition}, and to remove ordinarity one would  have to use  
Coleman families  and work on the eigenvariety.  

\end{document}